\documentclass{amsart}
\usepackage[utf8]{inputenc}

\usepackage{amssymb,comment}
\usepackage{mathrsfs}
\usepackage{stmaryrd}
\usepackage{mathabx}
\usepackage{amsthm}

\usepackage{tikz-cd}

\usepackage{lmodern}
\usepackage{csquotes}

\usepackage[T1]{fontenc}

\definecolor{darkgreen}{rgb}{0,0.5,0}
\definecolor{darkblue}{rgb}{0,0,0.8}
\definecolor{darkred}{rgb}{0.8,0,0}
\definecolor{lightblue}{rgb}{0,0.6,0.8}

\usepackage[pdfencoding=auto,colorlinks,citecolor=darkgreen,linkcolor=darkblue,urlcolor=darkred]{hyperref}

\usepackage[english]{babel}
\usepackage{url}
\usepackage[
backend=bibtex,
style=alphabetic,
sorting=nyt,
backref=false,
maxitems=10,
minitems=10,
maxbibnames=10,
minbibnames=10,
maxcitenames=10,
mincitenames=10,
giveninits=true,
maxalphanames=7,
minalphanames=7,
doi=false,
isbn=false,
url=false
]{biblatex}
\bibliography{references}

\usepackage{enumitem}
\usepackage{colonequals}

\DeclareFontFamily{U}{wncy}{}
\DeclareFontShape{U}{wncy}{m}{n}{<->wncyr10}{}
\DeclareSymbolFont{mcy}{U}{wncy}{m}{n}
\DeclareMathSymbol{\Sha}{\mathord}{mcy}{"58}

\usepackage{microtype}

\theoremstyle{plain}
\newtheorem{theorem}{Theorem}[subsection]
\newtheorem{lemma}[theorem]{Lemma}
\newtheorem{corollary}[theorem]{Corollary}
\newtheorem{proposition}[theorem]{Proposition}
\newtheorem{conjecture}[theorem]{Conjecture}

\newtheorem{remark}[theorem]{Remark}

\theoremstyle{definition}
\newtheorem{definition}[theorem]{Definition}

\theoremstyle{remark}

\usepackage{cleveref}
\crefname{theorem}{Theorem}{Theorems}
\crefname{lemma}{Lemma}{Lemmata}
\crefname{corollary}{Corollary}{Corollaries}
\crefname{proposition}{Proposition}{Propositions}
\crefname{definition}{Definition}{Definitions}
\crefname{conjecture}{Conjecture}{Conjectures}
\crefname{question}{Question}{Questions}
\crefname{example}{Example}{Examples}
\crefname{algorithm}{Algorithm}{Algorithms}
\crefname{remark}{Remark}{Remarks}
\crefname{assumption}{Assumption}{Assumptions}

\def\ol#1{\overline{#1}}% 		overline
\def\Alphabet{A,B,C,D,E,F,G,H,I,J,K,L,M,N,O,P,Q,R,S,T,U,V,W,X,Y,Z}%  Capitalized Alphabet
\def\alphabet{a,b,c,d,e,f,g,h,i,j,k,l,m,n,o,p,q,r,s,t,u,v,w,x,y,z}%	lowercase alphabet
\def\endpiece{xxx}%									marks end of list
\def\makeAlphabet[#1]{\expandafter\makeA#1,xxx,}%		Ex. \makeAlphabet[A,B]
\def\makealphabet[#1]{\expandafter\makea#1,xxx,}%		Ex. \makealphabet[c,d]
\def\makeA#1,{\def\temp{#1}\ifx\temp\endpiece\else%
	\mkbb{#1}\mkfrak{#1}\mkbf{#1}\mkcal{#1}\mkscr{#1}\mkbs{#1}\expandafter\makeA\fi}%
\def\makea#1,{\def\temp{#1}\ifx\temp\endpiece\else\mkfrak{#1}\mkbf{#1}\mkbs{#1}\expandafter\makea\fi}%
\def\mkbb#1{\expandafter\def\csname bb#1\endcsname{\mathbb{#1}}}%      Define bb
\def\mkfrak#1{\expandafter\def\csname fr#1\endcsname{\mathfrak{#1}}}%    Define frak
\def\mkbf#1{\expandafter\def\csname b#1\endcsname{\mathbf{#1}}}%           Define bold letters
\def\mkcal#1{\expandafter\def\csname c#1\endcsname{\mathcal{#1}}}%       Define calligraphy
\def\mkscr#1{\expandafter\def\csname s#1\endcsname{\mathscr{#1}}}%       Define script
\def\mkbs#1{\expandafter\def\csname bs#1\endcsname{{\boldsymbol{#1}}}}%       Define bold symbol
\def\makeop[#1]{\xmakeop#1,xxx,}%					Ex. \makeop[Hom,Spec]
\def\mkop#1{\expandafter\def\csname #1\endcsname{{\mathrm{#1}}}} % 
\def\xmakeop#1,{\def\temp{#1}\ifx\temp\endpiece\else\mkop{#1}\expandafter\xmakeop\fi}%
\def\makeup[#1]{\xmakeup#1,xxx,}%					Ex. \makeop[Hom,Spec]
\def\mkup#1{\expandafter\def\csname #1\endcsname{{\mathrm{#1}\,}}} % 
\def\xmakeup#1,{\def\temp{#1}\ifx\temp\endpiece\else\mkup{#1}\expandafter\xmakeup\fi}%
\makeAlphabet[\Alphabet]
\makealphabet[\alphabet]			
\makeop[AJ,CH,cores,cris,dR,Hom,Tor,Ext,Fil,holim,hocolim,dim,Sel,GL,SL,H,ord,Sym,Gal,Ann,ind,Aut,End,cd,Frob,Gr,sp,Tot,sm,an,Pic,NS,Br,tors,Reg,sgn,ss,nr,tr,new,cts,ur,alg,pd,disc,cyc,rk,cork,ab,nr,tors,Iw,res,inf,Cl,Quot,corank,Heeg,BK,AJ,div,MSD,Tam,loc,Nek,MSD,ac,BDP,PD,MTT,sgn,cond,prim,unr,rel,str,proj]

\makeup[Spec,Proj,Spwf,Sp,Sh,Spf,Sch,id,dR,rig,cris,im,coker]

\newcommand{\F}{\mathbf{F}}

\newcommand{\Q}{\mathbf{Q}}

\renewcommand{\epsilon}{\varepsilon}
\renewcommand{\theta}{\vartheta}
\renewcommand{\phi}{\varphi}

\newcommand{\mathup}[1]{\text{\textup{#1}}}
\renewcommand{\H} {\ensuremath{\mathup{H}}}

\newcommand{\Char}{\operatorname{Char}}

\newcommand{\defeq}{\colonequals}

\newcommand{\iso}{\simeq}

\numberwithin{equation}{section}

\begin{document}
\title[Iwasawa invariants of quadratic admissible twists]{On the Iwasawa invariants of admissible quadratic twists of elliptic curves at ordinary Eisenstein primes}
\author{Mulun Yin}
\address{(M. Yin) Morningside Center of Mathematics, No.55, Zhongguancun East Road, Beijing, 100190, China}
\email{mulunyin@amss.ac.cn}
\date{\today}

\begin{abstract}
	We study the Iwasawa $\mu$- and $\lambda$-invariants of quadratic twists of rational elliptic curves at Eisenstein primes $p\geq 3$. We verify Greenberg's $\mu=0$ conjecture for all such twists, which gives an Euler system free proof of Mazur's cotorsion conjecture, and study the $\lambda$-invariants when the twist is admissible in the sense of~\cite{GV00}. As a consequence, we prove the cyclotomic Iwasawa Main Conjecture and obtain the $p$-part BSD formulae in rank $0$ for admissible twists under mild conditions, covering many elliptic curves of potentially ordinary reduction.
\end{abstract}

\maketitle

\tableofcontents

\section*{Introduction}
\renewcommand{\thetheorem}{\Alph{theorem}}

\subsection{Background}In the ground-breaking work~\cite{GV00}, Greenberg and Vatsal studied the Iwasawa $\mu$- and $\lambda$-invariants of the Selmer group $\Sel(E/\bQ_\cyc)$ and the $p$-adic $L$-function $\cL_E$ of an elliptic curve, which combined with one divisibility proved by Kato~\cite{Kato}, yield the full Iwasawa Main Conjecture at Eisenstein primes of good ordinary reduction in many cases. Their arguments easily extend to the multiplicative reduction case, as long as their assumptions on the semisimplification of $E[p]$ are satisfied. More precisely, for an ordinary odd prime $p$ such that $E[p]$ is reducible as a $G_\bQ$-module (i.e., $p$ is an \emph{Eisenstein} prime for $E$), there is a short exact sequence of $G_\bQ$-modules\[
0\to \F(\phi) \to E[p]\to \F(\psi)\to 0,\]
where $\phi\psi=\omega$, the mod-$p$ cyclotomic character. One of $\phi$ and $\psi$ is odd (\emph{resp. ramified}), and the other is even (\emph{resp. unramified}). If the unramified character is odd, then Greenberg and Vatsal showed that \begin{itemize}
	\item $\mu(\Sel(E/\bQ_\cyc)^\vee)=\mu(\cL_E)=0$, and
	\item $\lambda(\Sel(E/\bQ_\cyc)^\vee)=\lambda(\cL_E)$.
\end{itemize}
These, combined with Kato's divisibility $\Char_\Lambda(\Sel(E/\bQ_\cyc)^\vee)\supset (\cL_E)$, yield the full Iwasawa Main Conjecture\[
\Char_\Lambda(\Sel(E/\bQ_\cyc)^\vee)=(\cL_E)\text{ in }\Lambda, 
\]
where $\Lambda=\bZ_p\llbracket \Gal(\bQ_\cyc/\bQ)\rrbracket\iso \bZ_p\llbracket T\rrbracket$.

More generally, Kato proved one divisibility over $\bQ(\mu_{p^\infty})$, which can be thought of as the cyclotomic $\bZ_p$-extension of $\bQ(\mu_p)$. There he simultaneously proved one divisilibity for each isotypic part according to the action of $\Delta=\Gal(\bQ(\mu_p)/\bQ)$:\begin{equation}\label{Katoomega^i}
\Char_\Lambda(\Sel(E\otimes \omega^i/\bQ_\cyc)^\vee)\supset (\cL(E,\omega^i)).
\end{equation}
In this paper, we extend the comparison of Iwasawa invariants in Greenberg--Vatsal to certain quadratic twists that are \emph{admissible} (see~\cref{defadm} for precise definitions). With the observations that\begin{itemize}
	\item Twisting by $\omega^i$ corresponds to taking the $\omega^i$-isotypic component~\eqref{Katoomega^i},
	\item Twisting by quadratic ramified characters turn ordinary reduction into potentially ordinary reduction, whose converse is true for potentially multiplicative reduction,
	\end{itemize}
we are able to get full Iwasawa Main Conjectures for certain quadratic twists of ordinary elliptic curves, and with an extended control theorem for the twists, to obtain rank $0$ BSD formula for some elliptic curves of additive reduction. In the process, we also discuss Greenberg's $\mu=0$ conjecture for all potentially ordinary elliptic curves coming from quadratic twists of ordinary ones, including all potentially multiplicative cases, and discuss an application in Mazur's cotorsion conjecture.

\subsection{The main results}\label{results}
The main results of this paper are the following.

\begin{theorem}(Greenberg's $\mu=0$ conjecture)
	Let $E$ be an elliptic curve of ordinary reduction at $p\geq 3$ be a prime, such that $E[p]$ is reducible as a $G_\bQ$-module. Let $\chi$ be a quadratic character corresponding to a quadratic field $K_\chi$. Then there exists an elliptic curve $E'$ in the $\bQ$-isogeny class of $E^{K_\chi}$ such that $\mu(\Sel_{\Gr_0}(E'/\bQ_\infty)^\vee)=0$.
\end{theorem}
This is the main result of~\cref{nonadGr}. In particular, when $\chi$ is ramified, $E^{K_\chi}$ has additive potentially ordinary reduction, and all curves with potentially multiplicative reduction arise in this way. The proof is analogous to that in~\cite{LY26}, which essentially covers all unramified quadratic twists.

An interesting consequence of this $\mu=0$ conjecture in the Eisenstein case is the following form of Mazur's cotorsion conjecture.

\begin{theorem}
		Assume $E/\bQ$ has ordinary reduction at an Eisenstein prime $p$ and $\chi$ is a quadratic ramified character. Then the Selmer group for the quadratic twist $E^\chi$ is $\Lambda$-cotorsion over any abelian number field $F$, i.e., Mazur's cotorsion conjecture holds in these cases.
\end{theorem}
This is~\cref{Mazur}. Similar results have been obtained in the literature using Euler systems, but our $\mu=0$ result provides an Euler system free method.

When the twist is admissible, we are also able to study the $\lambda$-invariants, in the style of~\cite{GV00}. A main consequence is the following $p$-part BSD formula in rank $0$.
\begin{theorem}($p$-part BSD in rank $0$)
Let $E$ be an elliptic curve of ordinary reduction at $p\geq 3$ such that $E[p]$ is reducible, and let $\chi$ be an admissible quadratic ramified character corresponding to a quadratic field $K_\chi$. Assume that $L(E^{K_\chi},1)\ne 0$ for the quadratic twist $E^{K_\chi}$, and that $E'(\bQ)[p]=0$ for all curves $E'$ isogenous to $E^{K_\chi}$. Then\[
\ord_p\bigl(\frac{L(E^{K_\chi},1)}{\Omega_{E^{K_\chi}}}\bigr)=\ord_p\bigl(\#\Sha(E^{K_\chi}/\bQ)\cdot\Tam(E^{K_\chi}/\bQ)\bigr),
\]
i.e., the rank $0$ $p$-part BSD formula for $E^{K_\chi}$ holds.

\end{theorem}
This is~\cref{rank0BSD}. In particular, this covers many elliptic curves of additive potentially ordinary reduction, including all curves with potentially multiplicative reduction whose isogeny class see no global torsion (in other words, $\F(\mathbf{1})$ is neither a $G_\bQ$-submodule nor a quotient of $E^{K_\chi}[3]$) at $p=3$.

The proof of the above theorem is a consequence of the cyclotomic Iwasawa Main Conjecture for the twists, as well as an explicit computation based on a control theorem. More precisely, we have the following.

\begin{theorem}(Cyclotomic Iwasawa Main Conjecture)
	Let $E$ be an elliptic curve with ordinary reduction at $p$ such that $E[p]$ is reducible over $G_\bQ$. Let $\chi$ be a ramified quadratic admissible twist of $E$. Assume $\H^0(\bQ,E'[p]\otimes\chi)=0$ for all $E'$ isogenous to $E$. Then the Iwasawa Main Conjecture for the twist $E^{K_\chi}$ is true, i.e.,\[
	\Char_{\Lambda}(\Sel(E^{K_\chi}/\bQ_\cyc)^\vee)=(\cL(E,\chi)).\]
	
\end{theorem}
This is~\cref{IMC}. The assumption that no curve in the isogeny class should have global torsion comes from the lack of knowledge for classical Iwasawa theory of the trivial representation $\F(\mathbf{1})$. It's connected to the BSD formula via the branched $p$-adic $L$-function developed in the work of Kato~\cite[section 16]{Kato} as well as the following calculation of Euler characteristic.

\begin{theorem}
	Assume $E/\bQ$ has ordinary reduction at $p\geq 3$ and that $\chi$ is a quadratic ramified character. If $\Sel_{\Gr_0}(E\otimes\chi/\bQ)$ is finite, then $\Sel_{\Gr_0}(E\otimes\chi/\bQ_\cyc)^\vee$ is generated by an $f_{E\otimes\chi}(T)\in\Lambda=\bZ\llbracket T\rrbracket$ not divisible by $T$ with\[
	f_{E\otimes\chi}(0)\sim\bigl(\prod_{p\ne v\mid N(E^{K_\chi})}\ord_p(c_{v,E^{K_\chi}})\bigr)\vert\Sel(E^{K_\chi}/\bQ)\vert/\vert E^{K_\chi}(\bQ)[p^\infty]\vert^2.
	\]
\end{theorem}
This is~\cref{fE(0)}. The formula is general and allows torsion. The above results show that despite there is no satisfying $p$-adic $L$-function for primes of additive reduction, in the case of potentially ordinary reduction, using twists is probably one suitable approach in deriving the BSD formula.

\subsection{Methods of proof and outline of the paper}
Let us briefly explain the ideas that go into the proof of our main results. The main motivation of the paper is that, studying twists of ordinary elliptic curves helps to understand certain elliptic curves of potentially ordinary reduction, which has been less touched in the literature. More precisely, if we twist a $p$-ordinary elliptic curve $E$ by a quadratic field $K$ where $p$ \emph{ramifies}, the twisted curve $E^K$ will have additive potentially ordinary reduction at $p$. In practice, we will very often view the quadratic twist from a representation theoretic point of view, i.e., we identify\[
\rho_{E^K}\iso\rho_E\otimes\chi_K\]
for the quadratic ramified character $\chi$ associated to $K$. Both viewpoints, together with a $\pm$-decomposition of modules for $E/K$ into direct sums of those for $E/\bQ$ and $E^K/\bQ$, are crucial to our arguments.

First, in term of the vanishing of $\mu$-invariants, we can actually cover all quadratic twists of ordinary elliptic curves. We do follow the methods in~\cite{GV00}, but the reason that we can go beyond their limited cases is that there is more flexibility in just showing the vanishing of $\mu$-invariants for $E[p]$ via the characters $\phi,\psi$ in $E[p]^{ss}$, with $\phi\psi=\omega$, the mod $p$ cyclotomic character. More precisely, assume $E[p]$ fits into the short exact sequence of $G_\bQ$-modules
\[0\to \F(\phi) \to E[p]\to \F(\psi)\to 0.\]
Greenberg has defined the Selmer group for $E$ using ordinary filtration, and a compatibility check shows that one can expect an exact sequence\begin{equation}\label{Etochar}
\Sel_?(\F(\phi)) \to \Sel_\ord(E[p]) \to \Sel_?(\F(\psi)),
\end{equation}
for suitable Selmer groups of $\phi$ and $\psi$. Indeed, one should use the relaxed Selmer group for the ramified character and the strict Selmer group for the unramified character. If one takes a twist by $\chi$, then one should use the same Selmer groups\[
\Sel_?(\F(\phi\chi)) \to \Sel_\ord(E[p]\otimes\chi) \to \Sel_?(\F(\psi\chi)),
\] for $\phi\chi$ (\emph{resp. $\psi\chi$}) and $\phi$ (\emph{resp. $\psi$}), even though ramification can change after twist. Since the vanishing of $\mu$-invariant for $\Sel_\ord(E[p^\infty]\otimes\chi/\bQ_\cyc)$ is equivalent to the vanishing of $\mu$-invariant for $\Sel_\ord(E[p]\otimes\chi)$, it suffices to relate $\mu(\Sel_\ord(E[p]\otimes\chi))$ to $\mu(\Sel_?(\F(\phi\chi)))$
and $\mu(\Sel_?(\F(\psi\chi)))$.

Classical Iwasawa theory, more specifically Mazur--Wiles's theorem, says that if a (non-trivial) character is even, then its \emph{relaxed} Selmer group has vanishing $\mu$, whereas if a (non-cyclotomic) character is odd, then its \emph{strict} Selmer group has vanishing $\mu$. The condition in~\cite{GV00} is exactly that the Selmer group $\Sel_?$ in~\eqref{Etochar} determined by ramification agrees with the Selmer group determined by parity, i.e.,\begin{itemize}
	\item The unramified character should be odd after twisting by $\chi$ (using $\Sel_{str}$),
	\item The ramified character should be even after twisting by $\chi$ (using $\Sel_{\rel}$).
\end{itemize}
This is exactly what we call `admissibility' from~\cite{GV00}. But, as one may notice, Mazur--Wiles's theorem applies to all strict Selmer groups, so the only way one can see a non-trivial $\mu$ is if one uses relaxed Selmer group for a twisted character that becomes odd and ramified.

Even in the exceptional case, Ribet's lemma allows us to choose an isogenous curve $E'$ (respected by quadratic twists) such that the odd and ramified character is a quotient, rather than a submodule of $E'\otimes\chi$. Then one can use a slight generalization of the methods developed in~\cite{LY26} to show that even in this case $\Sel_\ord(E'[p]\otimes\chi)$ will still have a vanishing $\mu$-invariant, despite a positive $\mu$ in some quotient. This confirms Greenberg's $\mu=0$ conjecture for many potentially ordinary elliptic curves. We also discuss how this $\mu=0$ result can be used to verify Mazur's cotorsion conjecture.

The study of $\lambda$-invariants is much more delicate. Now the admissibility condition in~\cite{GV00} becomes crucial in deducing the relations\begin{itemize}
	\item $\lambda(\Sel^\Sigma_\ord(E[p]\otimes\chi)^\vee)=\lambda(\Sel^\Sigma_?(\F(\phi\chi))^\vee)+\lambda(\Sel^\Sigma_?(\F(\psi\chi))^\vee)$,
	\item $\lambda(\cL^\Sigma(E,\chi))=\lambda(\cL^\Sigma(\phi\chi))+\lambda(\cL^\Sigma(\phi\chi)))$
\end{itemize}
 between the imprimitive objects on the algebraic and analytic sides, since otherwise the above equations are not expected to hold. Indeed, on the algebraic side, for the characters, there will be a mismatch between the Selmer groups determined by ramification and those by parity, and Mazur--Wiles's theorem could not be applied to the appropriate Selmer groups. On the analytic side, non-admissibility would break congruence between certain modular forms due to bad behavior of canonical periods. There doesn't seem to be a reasonable way to compare Iwasawa invariants in the non-admissible case.
 
 Therefore, we content ourselves with just comparing Iwasawa invariants in the admissible case. This is still beyond the scope of~\cite{GV00} since they only studied the trivial twist. However, all the ingredients are essentially available in their work. We still briefly explain their main idea for completeness.
 
 First, Kato proved one divisibility of the cyclotomic Iwasawa Main Conjecture\[\Char_{\Lambda}(\Sel_{\Gr_0}(E\otimes\omega^i/\bQ_\cyc)^\vee)\supset (\cL(E,\omega^i))
 \]
 for an ordinary elliptic curve twisted by powers of $\omega$. The idea in~\cite{GV00} is that if one knows the Iwasawa $\lambda$- and $\mu$-invariants of both sides are equal, then this divisibility must in fact be an equality. They achieved this by showing the $\mu$-invariants of both sides are zero, and that a Main Conjecture is equivalent to its `imprimitive' version. More precisely, it suffices to show\[
 \lambda(\Sel^\Sigma_\ord(E[p]\otimes\chi)^\vee)=\lambda(\cL^\Sigma(E,\chi)).
 \]
 Without loss of generality, we may take $\chi$ to be the quadratic character $\omega^{\frac{p-1}{2}}$, since the unramified part can be absorbed in $E$ (twisted to another ordinary curve if necessary). Now Mazur--Wiles's theorem tell us that\begin{itemize}
 	\item $\lambda(\Sel^\Sigma_?(\F(\phi\chi))^\vee)=\lambda(\cL^\Sigma(\phi\chi))$,
 	\item $\lambda(\Sel^\Sigma_?(\F(\psi\chi))^\vee)=\lambda(\cL^\Sigma(\psi\chi))$,
 \end{itemize}
 for appropriate choice of Selmer groups, as long as $\phi\chi$ and $\psi\chi$ are not the trivial character. Together with the relation between the $\lambda$-invariants of $E\otimes\chi$ and $\phi\chi$, $\psi\chi$ established earlier, we can deduce the Iwasawa Main Conjecture
 \[\Char_{\Lambda}(\Sel_{\Gr_0}(E\otimes\omega^i/\bQ_\cyc)^\vee)=(\cL(E,\omega^i))
 \]as in~\cite{GV00}.
 
 It is worth mentioning that, if $E$ has ordinary reduction, then the non-admissible case (with trivial twist) can be studied using Beilinson--Flach elements in a $\bZ_p^2$-extension of an imaginary quadratic field. One can avoid the comparison of Iwasawa invariants and still obtain the cyclotomic Iwasawa Main Conjecture from its anticyclotomic counterpart. See~\cite{CGS},~\cite{KY24}. It is not clear to us at the moment if this method is compatible with taking twists.
 
 Finally, we explain how to go from the Iwasawa Main Conjecture to the $p$-part BSD formula. It naturally boils down to the comparison of a calculation of Euler characteristic on the algebraic side with the interpolation on the analytic side. Usually, one would need to construct a $p$-adic $L$-function for the curve $E^K$, but $p$-adic $L$-function for elliptic curves in the additive reduction case is extremely hard. We attempt to realize the computation using the twist $E\otimes\chi$, and show that the algebraic side and the analytic side do combine beautifully and lead to the $p$-part BSD formula in rank $0$ for $E^K$ having additive potentially ordinary reduction.
 
\subsection{Relation to previous works}
This work is a natural extension of~\cite{GV00}, where all the main ideas originate from, including the comparison of Iwasawa invariants in the admissible settings. We use the branched $p$-adic $L$-functions as well as one divisibility in Kato's work~\cite{Kato}, and the calculation of Euler characteristic is inspired from Greenberg's work~\cite{Greenberg1999}. 

In the non-admissible settings, the computation of $\mu$-invariant is inspired by the method developed in~\cite{LY26}. 

This work will complement a forthcoming work of the author~\cite{XY26}, which focuses on potentially good reduction.

\subsection{Notations}
For any subextension $L/\Q$ of $\bar{\bQ}/\bQ$, we let $G_L\coloneq \Gal(\bar{\bQ}/L)$ denote its absolute Galois group. For a cohomology group $\H^i(\cdot,-)$, we write $\H^i(L,-)$ in place of $\H^i(G_L,-)$. Denote by $(-)^\vee=\Hom_\cts(-,\bQ_p/\bZ_p)$ the Pontryagin dual. 

\subsection{Future Works}
In this work, only quadratic twists of ordinary elliptic curves are considered, covering most potentially multiplicative reduction and many potentially good ordinary reduction. It would be natural to see if the same arguments apply to more general twists, hoping to cover all potentially ordinary cases. 

On the other hand, one can ask if the ideas in~\cite{CGS} can be adapted using twists to attack the non-admissible setting. We would also hope to study this in a future work.

Computing the $\mu$-invariants for general twists and checking Greenberg's $\mu=0$ conjecture at all potentially ordinary Eisenstein primes will also be studied in a future project.

\renewcommand{\thetheorem}{\arabic{section}.\arabic{subsection}.\arabic{theorem}}

\section{Preliminary}\label{preliminary}
In this section, we record relevant facts about quadratic twists of ordinary elliptic curves and modular forms.

\subsection{Ordinary reduction at Eisenstein primes}
Let $E$ be an elliptic curve. We say $p$ is an \emph{Eisenstein} prime for $E$ is $E[p]$ is reducible as a $G_\bQ$-module. In other words, there is an exact sequence of $G_\bQ$-modules\begin{equation}\label{red}
0\to \F(\phi)\to E[p]\to \F(\psi)\to 0
\end{equation}
for $\phi,\psi:G_\bQ\to\Aut(\F)$, where $\F(-)$ denotes the module $\F$ with $G_\bQ$-action given by the character. Following Greenberg's convention~\cite{Greenberg1989}, we say $p$ is an \emph{ordinary} prime for $E$ if $A\defeq E[p^\infty]$ has a ramified $G_{\bQ_p}$-submodule $\Fil^+(A)$ of $\bZ_p$-corank $1$ with an unramified quotient $\Fil^-(A)\defeq A/\Fil^+(A)$. Let $f$ be the weight $2$ newform associated to $E$. Then one knows that $E$ is ordinary at $p$ if and only if the $p$-th Fourier coefficient $a_p(f)$ of $f$ is a $p$-adic unit. By a result of Mazur~\cite{Mazur}, good Eisenstein primes are ordinary. In particular, supersingular reduction cannot exist at Eisenstein primes, and ordinary primes are exactly the non-additive primes.

The ordinary filtration naturally propagates to the $p$-torsion level. That is, there is a filtration on $E[p]$, so that it has a ramified one dimensional $G_{\bQ_p}$-sub-representation $\Fil^+(E[p])$ with an unramified one-dimensional quotient $\Fil^-(E[p])\defeq E[p]/\Fil^+(E[p])$. For Eisenstein primes, it is clear that one of $\phi,\psi$ in~\eqref{red} will be ramified (with the other unramified), and the ramified part restricts to $\Fil^+(E[p])$ over $G_{\bQ_p}$. As $\phi\psi=\omega$, the ramified character is tamely ramified.

\subsection{Quadratic twists of newforms and elliptic curves}\label{twistvsK}
Let $f\in S_2^\new(\Gamma_0(N))$ be a newform of weight $2$. Then its twist by a quadratic Dirichlet character $\chi$ is another newform $f_\chi\defeq f\otimes \chi\in S_2^\new(\Gamma_0(N'))$, where $N'$ divides $N\cdot \cond(\chi)^2$, provided that $\chi$ is primitive. In general, twisting by a non-primitive character will yield an oldform. If $\chi$ is ramified at $p$, i.e., $p\mid \cond(\chi)$, then the level of $f_\chi$ will be divisible by $p^2$, and $f_\chi$ will have additive reduction at $p$:\begin{itemize}
	\item If $f$ has good ordinary reduction at $p$, $f_\chi$ has potentially good ordinary reduction at $p$;
	\item If $f$ has multiplicative reduction at $p$, $f_\chi$ has potentially multiplicative reduction at $p$.
\end{itemize} 
The following provides a partial converse of the above statement.
\begin{proposition}
	Let $E/\bQ$ be an elliptic curve of potentially ordinary reduction at $p$ with associated newform $f$, then there is a Dirichlet character $\chi$ such that $f\otimes \chi^{-1}$ is ordinary at $p$.
\end{proposition}
\begin{proof}
 This is~\cite[Theorem 3.1.1]{KY24b}, which is a consequence of~\cite[Proposition 12.11.5(iv)]{NekovarSelmerComplexes}.
\end{proof}
In particular, if $f$ corresponds to an elliptic curve of potentially multiplicative reduction, then $\chi$ can be taken to be a quadratic character, since a general twist will multiply the Nebentypus by its square, and newforms corresponding to multiplicative reduction must have unramified Nebentypus. Indeed, any ramified character $\theta:G_\bQ\to\F^\times$ (since $f$ has rational coefficients) factors into a power of the mod $p$-cyclotomic character and an unramified character, and one can take $\chi$ to be the cyclotomic part, as unramified twists preserve ordinariness, which then must obviously be quadratic. One should notice that this proposition does not claim the twist is necessarily primitive, so $f\otimes\chi^{-1}$ itself might not be a newform. However, writing $\tilde{f}$ for the corresponding newform, we still have the Galois representation theoretic equality \[\rho_f=\rho_{\tilde{f}}\otimes\chi,\]
where by an abuse of notation $\chi$ denotes $\rho_\chi$, the Galois representation attached to $\chi$. Thus, the newform associated to an elliptic curve of potentially multiplicative reduction always comes from a quadratic twist of a $p$-multiplicative newform in the Galois theoretic sense.

Note that each non-trivial quadratic character $\chi$ corresponds to a quadratic field $K_\chi$ over which it becomes trivial. More precisely, if $\chi_\prim:G_\bQ\to\{\pm 1\}$ denotes the primitive quadratic character associated to $\chi$, then $K_\chi=\bQ^{\ker(\chi_\prim)}$. Writing $K_\chi=\bQ(\sqrt{D})$, then $\chi_\prim$ agrees with the Kronecker character $(\frac{D}{\cdot})$ with conductor $\vert D\vert$, and one knows $\rho_{\chi_\prim}=\rho_\chi$. If $f$ is associated to an elliptic curve $E$ with potentially multiplicative reduction at $p$, then $E^{K_\chi}=E^{(D)}$ corresponds to $\rho_f\otimes \rho_{\chi_\prim}=\rho_f\otimes\chi=\rho_f\otimes \chi^{-1}$ and is an elliptic curve of multiplicative reduction at $p$. Thus we have showed the following.

\begin{proposition}\label{untwist}
	\begin{enumerate}
		\item Twisting an elliptic curve $E$ by an quadratic character $\chi$ in the sense of Galois representation is the same as the literal quadratic twist of $E$ by the quadratic field $K_\chi$ associated to $\chi$. An unramified twist preserves ordinary reduction and a ramified twist changes ordinary reduction to potentially ordinary reduction which becomes ordinary over $K_\chi$.
		\item As a partial converse, let $E$ be an elliptic curve of potentially multiplicative reduction at $p$. Then there is a quadratic field $K$ such that its quadratic twist $E^K$ is an elliptic curve of multiplicative reduction at $p$.
	\end{enumerate}
	
\end{proposition}
The same conclusion that an elliptic curve of potentially multiplicative reduction gains ordinary reduction over a quadratic field can also be seen using classifications of Kodaira types.

\subsection{Admissible twists}\label{defadm}
We adopt the notion of admissible twists from~\cite{GV00}, which originates from the study of periods.

\begin{definition}\label{defadmissible}
	We say that a triple $(E,p,\chi)$ consisting of an elliptic curve $E$, an odd prime $p$ and a Dirichlet character $\phi$ is \emph{admissible} if:\begin{itemize}
		\item $E[p]$ is reducible and fits into an exact sequence\[
		0\to \F(\phi)\to E[p]\to \F(\psi)\to 0,\]
		for some characters $\phi,\psi:G_\bQ\to\Aut(\F_p)$ with $\phi\psi=\omega$, the mod-$p$ cyclotomic character. In particular, exactly one of $\phi$,$\psi$ is odd.
		\item $E$ is ordinary at $p$, i.e., locally over $G_{\bQ_p}$, $E[p]$ has a ramified subline (which may or may not be the restriction of $\F(\phi)$, as the above sequence might become split locally) with an unramified quotient. In particular, exactly one of $\phi$,$\psi$ is ramified.
		\item $\chi$ is unramified or tamely ramified at $p$ with $\theta(-1)=-\chi(-1)$ where $\theta$ is the unramified character in $\phi$ and $\psi$, i.e., $\chi$ has opposite parity from the unramified character.
		\end{itemize}
	When these hold, we shall simply say that $E\otimes\chi$ is an admissible twist of $E$. 
\end{definition}
This agrees with the assumption in~\cite{GV00} that \begin{equation*}\tag{GV}\label{GV}
\phi\text{ is either unramified and odd, or ramified and even}
\end{equation*} when they take $\chi$ to be the trivial character. Under the more general admissibility assumption, the periods behave well and there is a congruence\[
\cL(E,\chi)\equiv \cL_{\phi\chi}\cL_{\psi\chi}\pmod{p\Lambda}
\]
between the twisted $p$-adic $L$-functions of $E$ and those of the characters appearing in $E[p]$. These are discussed beautifully in~\cite[section 3]{GV00}.

In particular, if one takes $\chi$ to be a tamely ramified character, the above congruence allows us to compare Iwasawa invariants even in some additive reduction cases. 

From now on, we shall restrict our attention to admissible quadratic twists of ordinary elliptic curves.

\section{Iwasawa theory for elliptic curves of ordinary reduction}\label{2}
In this section, we recall Selmer groups and $p$-adic $L$-functions in Iwasawa theory for elliptic curves discussed in~\cite{GV00}. 

\subsection{Cyclotomic Iwasawa theory}
We say an algebraic extension of a number field is a $\bZ_p$-extension if the Galois group is isomorphic to the $p$-adic integers. There is a unique $\bZ_p$-extension of $\bQ$, the cyclotomic $\bZ_p$-extension $\bQ_\cyc$, obtained as $\bQ_\cyc\defeq\bQ(\mu_{p^\infty})^{\Gal(\bQ(\mu_p)/\bQ)}$, where $\mu_p$ (resp. \emph{$\mu_{p^\infty}$}) denotes the $p$-th (resp. \emph{$p$-power-th}) roots of unity. For any number field $F$, denote by $F_\cyc\defeq F\bQ_\cyc$ its cyclotomic $\bZ_p$-extension with Galois group $\Gamma_F\defeq \Gal(F_\cyc/F)\iso\bZ_p$, and let $\Lambda_F\defeq \bZ_p\llbracket \Gamma_F\rrbracket$ be the corresponding Iwasawa algebra, which is usually identified with the formal power series ring $\bZ_p\llbracket T\rrbracket$ when a topological generator $\gamma$ of $\Gamma_F$ is identified with $1+T$. The $p$-adic $L$-functions of an elliptic curve to be recalled will be elements of $\Lambda_F$, while the Selmer groups are known to be cofinitely generated and cotorsion over $\Lambda_F$, meaning their Pontryagin duals are finitely generated and torsion. For a torsion $\Lambda_F$ module $X$ there is a $\Lambda_F$-morphism with finite kernel and cokernel\[
X\to \bigoplus_{i=1}^n \Lambda_F/(f_i)\oplus \bigoplus_{j=1}^m \Lambda_F/(p^{m_j}),
\]
where $f_i$ are \emph{distinguished polynomials} with $f_i(T)\equiv T^{\deg(f_i)}\pmod{p}$, and $m_j$ are non-negative integers. The values $\lambda(X)\defeq\sum_i\deg(f_i)$ and $\mu(X)\defeq\sum_j m_j$ are called the Iwasawa $\lambda$- and $\mu$-invariants of $X$ respectively. Define the characteristc polynomial of $X$ to be $p^{\mu(X)}\prod_i f_i$, and the \emph{characteristic ideal} $\Char_{\Lambda_F}(X)$ of $X$ to be the $\Lambda$ ideal generated by its characteristic polynomial.

If two torsion $\Lambda_F$-modules $X_1$ and $X_2$ satisfy $\Char_{\Lambda_F}(X_1)\subset \Char_{\Lambda_F}(X_2)$ and $\mu(X_1)=\mu(X_2)$, $\lambda(X_1)=\lambda(X_2)$, then it follows that $\Char_{\Lambda_F}(X_1)=\Char_{\Lambda_F}(X_2)$.

\subsection{The algebraic side: Selmer groups}
In this section, we define the Selmer groups and discuss some of its important properties. The goal is to describe the Iwasawa theoretic results on the algebraic side which will be compared to those on the analytic side in the next section. We will restrict our attention to the rational numbers $\bQ$ and certain quadratic fields $K$.

Let $F$ be $\bQ$ or $K$. Let $\Sigma$ be a finite set of places of $F$ containing those above $p$, those where $E$ has bad reduction and the infinite places. We define \emph{Greenberg's} Selmer group with ordinary local conditions for $E$ over $F_\cyc$, $\Sel_\Gr(E/F_\cyc)$, as the kernel of the global-to-local map\[
	\H^1(F^\Sigma/F_\cyc, E[p^\infty]) \to \prod_{v\mid p}\H^1(I_v,\Fil^-(E[p^\infty]))\times\prod_{v\nmid p}\H^1(F_{\cyc,v},E[p^\infty]) ,\]
where $F^\Sigma$ is the maximal extension of $F$ unramified outside $\Sigma$ and $I_v$ denotes the inertia subgroup of $G_{F_\cyc}$. The mapping to local at $p$ part is understood as the composition\[
\H^1(F^\Sigma/F_\cyc,E[p^\infty])\to \prod_{v\mid p}\H^1(F_{\cyc,v},E[p^\infty])\to \prod_{v\mid p}\H^1(I_v,\Fil^-(E[p^\infty]))
\]Its $\Sigma$-\emph{imprimitive} version is defined as\[
\Sel^\Sigma_\Gr(E/F_\cyc)=\ker\big(\H^1(F^\Sigma/F_\cyc, E[p^\infty]) \to \prod_{v\mid p}\H^1(I_v,\Fil^-(E[p^\infty]))\big).
\]

For later use, we also define the residual Selmer groups\[
	\Sel^\Sigma_\Gr(E[p]/F_\cyc)=\ker\big(\H^1(F^\Sigma/F_\cyc, E[p]) \to \prod_{v\mid p}\H^1(I_v,\Fil^-(E[p]))\big).
\]

\begin{remark}
When $E$ has good ordinary or non-split multiplicative reduction, one can replace $\prod_{v\mid p}\H^1(I_v,\Fil^-(E[p^\infty]))$ by $\prod_{v\mid p}\H^1(F_{\cyc,v},\Fil^-(E[p^\infty]))$(which gives what we call `strict' Greenberg's Selmer group and denote by $\Gr_0$) without affecting the Selmer groups. When $E$ has split multiplicative reduction and $F=\bQ$, however, the strict Selmer group is strictly smaller than Greenberg's Selmer group, and their difference corresponds to the `trivial zero' of the $p$-adic $L$-function in the next subsection.
\end{remark}

Finally, one can always define the usual $p$-infinity Selmer group $\Sel(E/F)$ coming from Kummer images for any algebraic field $F$,\[
\Sel(E/F)\defeq\ker\bigl(\H^1(F^\Sigma/F,E[p^\infty])\to\prod_v\frac{\H^1(F_{\cyc,v},E[p^\infty])}{\im(\kappa_v)}\bigr),
\]
and its imprimitive version
\[
\Sel^\Sigma(E/F)\defeq\ker\bigl(\H^1(F^\Sigma/F,E[p^\infty])\to\frac{\H^1(F_{\cyc,p},E[p^\infty])}{\im(\kappa_p)}\bigr).
\]

We now collect some well-known facts about the Selmer groups.
\begin{proposition}\label{Selinfo}
 \begin{enumerate}
 	\item If $E$ has ordinary reduction at all places above $p$, then the above strict Selmer groug $\Sel_{\Gr_0}(E/F_\cyc)$ agree with the usual Selmer group $\Sel(E/F_\cyc)$ defined via Kummer image.
 	\item The $\mu$-invariant of $\Sel_{\Gr_0}(E/F_\cyc)^\vee$ equals the $\mu$-invariant of $\Sel^\Sigma_{\Gr_0}(E/F_\cyc)^\vee$, and the difference between their $\lambda$-invariants can be explicitly computed.
 	\item The $\mu$-invariant of $\Sel^\Sigma_{\Gr_0}(E/F_\cyc)^\vee$ is $0$ if and only if the $\mu$-invariant of $\Sel^\Sigma_{\Gr_0}(E[p]/F_\cyc)^\vee$ is $0$.
 	\item Conjugations by the non-trivial element of $\Gal(K_\cyc/\bQ_\cyc)$ and the non-trivial element of $\Gal(K_{\cyc,p}/\bQ_{\cyc,p})$ split $\Sel^\Sigma(E/K_\cyc)$ into a direct sum of $\pm$-eigencomponents which yields a decomposition\[
 	\Sel^\Sigma(E/K_\cyc)=\Sel^\Sigma(E/\bQ_\cyc)\oplus\Sel^\Sigma(E^K/\bQ_\cyc).
 	\]
 	In particular, if $E$ has ordinary reduction at all primes above $p$, there is a decomposition for imprimitive strict Greenberg's Selmer group over $K$\[
 	\Sel^\Sigma_{\Gr_0}(E/K_\cyc)=\Sel^\Sigma(E/\bQ_\cyc)\oplus\Sel^\Sigma(E^K/\bQ_\cyc),
 	\]
 	where at least one of $\Sel^\Sigma(E/\bQ_\cyc)$ and $\Sel^\Sigma(E^K/\bQ_\cyc)$ is identified with the corresponding imprimitive strict Greenberg's Selmer group.
 \end{enumerate}
\end{proposition}
\begin{proof}
	\begin{enumerate}
		\item This is in~\cite[section 2]{Greenberg1989}.
		\item This follows from~\cite[Proposition 2.4]{GV00}, which still works for $K_\cyc$.
		\item This is~\cite[Proposition 3.1.3]{LY26}, which still works for $K_\cyc$.
		\item Write $\Gal(K_\cyc/\bQ_\cyc)=\{\id,\tau\}$ and $\Gal(K_{\cyc,p}/\bQ_{\cyc,p})=\{\id,\sigma\}$. Then the isomorphism $\Gal(K_\cyc/\bQ_\cyc)\iso \Gal(K_{\cyc,p}/\bQ_{\cyc,p})$ identifies $\tau$ with $\sigma$ which respects restriction maps \begin{center}{\small\begin{equation*}\begin{tikzcd}	
						\H^1(K^\Sigma/K_\cyc,E[p^\infty]) \ar[r,"="]\ar[d,"\res"]& \H^1(K^\Sigma/K_\cyc,E[p^\infty])^{\frac{1+\tau}{2}}\oplus\H^1(K^\Sigma/K_\cyc,E[p^\infty])^{\frac{1-\tau}{2}} \ar[d,"\res"]\\
						\H^1(K_{\cyc,p},E[p^\infty])\ar[r,"="]&\H^1(K_{\cyc,p},E[p^\infty])^{\frac{1+\sigma}{2}}\oplus\H^1(K_{\cyc,p},E[p^\infty])^{\frac{1-\sigma}{2}}\\	 	\end{tikzcd}
			\end{equation*}}
		\end{center}
	and Kummer images\[
	 \im\bigl((E(K_{\cyc,p})\otimes\bQ_p/\bZ_p)^{\frac{1\pm\sigma}{2}}\xrightarrow{\kappa_p} \H^1(K_{\cyc,p},E[p^\infty])^{\frac{1\pm\sigma}{2}}\bigr).\]
The splitness of $\Sel^\Sigma(E/K_\cyc)$ then follows from the definitions\[
\Sel^\Sigma(E/K_\cyc)^{\frac{1\pm\tau}{2}}\defeq\ker\bigl(\H^1(K^\Sigma/K_\cyc,E[p^\infty])^{\frac{1\pm\tau}{2}}\to\frac{\H^1(K_{\cyc,p},E[p^\infty])^{\frac{1\pm\sigma}{2}}}{\im(\kappa_p)^{\frac{1\pm\sigma}{2}}}\bigr).
\]
	The claim now follows from the observation that $\Sel(E/\bQ_\cyc)$ is the $+$-part $\Sel^\Sigma(E/K_\cyc)^{\frac{1+\tau}{2}}$ and $\Sel(E^K/\bQ_\cyc)$ is the $-$-part $\Sel^\Sigma(E/K_\cyc)^{\frac{1-\tau}{2}}$, both coming from the inflation-restriction exact sequence as $p\nmid [K:\bQ]=2$. 	
	Also note that if $E$ (or $E^K$) has ordinary reduction at all places above $p$ in $K$ but it has additive reduction over $\bQ$, then the bad reduction is killed by the Kronecker character $(\frac{D_K}{\cdot})$, so its quadratic twist will have ordinary reduction over $\bQ$. 
	\end{enumerate}	
\end{proof}

\begin{remark}
	We discuss a more general phenomenon regarding (iv) above. Since we are mostly interested in \emph{ramified} twists of ordinary elliptic curves, we will always take $K$ to be a quadratic extension where $p$ \emph{ramifies}, thus $[K:\bQ]=[K_p:\bQ_p]=[K_\cyc:\bQ_\cyc]=[K_{\cyc,p}:\bQ_{\cyc,p}]=2$, and the unique non-trivial elements in the degree $2$ Galois groups split all modules (e.g. $\H^i(-,M)$ for some $M$ of $p^\infty$-exponent) over $K$, $K_p$, $K_\cyc$ and $K_{\cyc,p}$ into $\pm$-eigencomponents, and respects all restriction maps (e.g. the global-to-local maps and finite level to infinity level maps). Kummer images also decompose into $\pm$-parts. In particular, the $+$-component $\H^1(K,M)^+$ is just $\H^1(\bQ,M)$ from the inflation-restriction sequence since $p\nmid 2$ (similar for the other quadratic extensions), and $\H^1(K,M)^-$ is $\H^1(\bQ,M(\chi_K))$ for the $K$-twisted module. This allows us to explicitly describe the Selmer groups for $E[p^\infty]\otimes\chi$ which is possibly no longer ordinary. 
\end{remark}

As an example of the remark above, we describe the Selmer group for elliptic curves of potentially ordinary reduction coming from quadratic twists of ordinary curves. Let $E$ be $p$-ordinary and $\chi$ a ramified quadratic character corresponding to a quadratic field $K$, so $E^K$ has additive potentially ordinary reduction. Then the Selmer group for $E$ over $\bQ$ and $K$ are both Greenberg's Selmer group, so globally there is a decomposition \[\H^1(K^\Sigma/K_\cyc,E[p^\infty])=\H^1(\bQ^\Sigma/\bQ_\cyc,E[p^\infty])\oplus\H^1(\bQ^\Sigma/\bQ_\cyc,E^K[p^\infty]),\] where the last term is identified with $\H^1(\bQ^\Sigma/\bQ_\cyc,E[p^\infty]\otimes\chi)$,
and locally there is a decomposition \[\H^1(K_{\cyc,p},\Fil^-(E[p^\infty]))=\H^1(\bQ_{\cyc,p},\Fil^-(E[p^\infty]))\oplus\text{(local term for $E^K$)}.\] One sees that $\H^1(\bQ_{\cyc,p},(\Fil^-(E[p^\infty]))\otimes \chi)$ shall be the correct local term defining $\Sel^\Sigma(E^K/\bQ_\cyc)$ yielding the splitting of Selmer group over $K$. This will be revisited in~\cref{2.5.3}.

Now assume $p$ is an Eisenstein prime for $E$, that is, $E[p]$ is reducible as a $G_\bQ$-module, and recall the theory for Selmer groups developed in~\cite{GV00}. First assume $E$ has ordinary reduction at $p$. Then the decomposition\[
0\to \F(\phi) \to E[p] \to \F(\psi) \to 0
\]
yields an exact sequence\begin{equation}\label{selEtochar}
\Sel^\Sigma(\F(\phi)/\bQ_\cyc) \to \Sel_\Gr^\Sigma(E[p]/\bQ_\cyc) \to \Sel^\Sigma(\F(\psi)/\bQ_\cyc)
\end{equation}
where the Selmer groups for a character $\theta$ are defined via unramified local conditions
	\[\Fil^+(\F(\theta))=\begin{cases}
		\F(\theta)& \text{if }\theta\text{ is ramified,}\\
		0& \text{if }\theta\text{ is unramified},
	\end{cases}
	\]
	and $\Fil^-(\F(\theta))=\F(\theta)/\Fil^+(\F(\theta)).$
More precisely, for $F=\bQ$ or $K$ and $\theta=\phi,\psi$, we have\[
\Sel^\Sigma(\F(\theta)/F_\cyc)\defeq\ker\bigl(\H^1(F^\Sigma/F_\cyc,\F(\theta))\to \prod_{v\mid p}\H^1(F_{\infty,v},\Fil^-(\F(\theta)))\bigr).\]

The sequence~\eqref{selEtochar} is often also exact on the left and right, such as when the assumption that the unramified character is odd in~\cite{GV00} is satisfied. In particular, under the same assumption, the $\mu$-invariant of $\phi$ and $\psi$ are $0$, so the $\mu$-invariant of $E$ is also $0$.

\subsection{The analytic side: $p$-adic $L$-functions}
In this section, we describe the cyclotomic $p$-adic $L$-function constructed by Mazur, Tate and Teitelbaum~\cite{MTT}. Let $E/\bQ$ be an elliptic curve and let $p$ be a prime of ordinary reduction for $E$. For a Dirichlet character $\chi$, denote by $L(E/\bQ,\chi,s)$ the Hasse--Weil $L$-series attached to $E$, twisted by $\chi$, and let $\Omega_E$ be the real N\'eron period for $E$. Let $\alpha_p$ be the unit root of the Hecke polynomial\[
x^2-a_p(E)x+p,
\]where $a_p(E)=p+1-\#\tilde{E}(\F)$. When $E$ has multiplicative reduction, take $\alpha_p=a_p(E)$.

Let $\Lambda=\Lambda_\bQ$. When $E$ has good ordinary reduction at $p$, Mazur and Swinnerton-Dyer~\cite{MSD74} have constructed an element $\cL_E^\MSD(T) \in\Lambda\otimes \bQ_p$, known as the (Mazur--Swinnerton-Dyer) $p$-adic $L$-function, satisfying the following interpolation property:
for any finite order character $\chi$ of $\Gamma$ of conductor $p^r$ with $r >0$, we have\[
\cL_E^\MSD(\chi)=\frac{p^r}{\tau(\ol\chi) \alpha_p^r}\cdot \frac{L(E/\bQ,\chi,1)}{\Omega_E},
\]
where $\tau(\ol\chi)=\sum_{a\pmod{p^r}}\ol\chi(a)e^{2\pi ia/p^r}$ denotes the Gauss sum, and \[
\cL_E^\MSD(\mathbf{1})=(1-\alpha_p^{-1})^2\cdot \frac{L(E/\bQ,1)}{\Omega_E}.
\]

\begin{lemma}[\cite{GV00},~\cite{Wut17}]
	$\cL_E^\MSD$ is integral, i.e., $\cL_E^\MSD\in\Lambda$.
\end{lemma}
\begin{proof}
	See for example~\cite[Theorem 2.1.1]{CGS}.
\end{proof}	
When $E$ has multiplicative reduction, the $p$-adic $L$-function was later refined in~\cite{MTT}, which is now also known to be integral, and we still denote it by $\cL^\MSD$. Note that when $\alpha_p=1$, i.e. when $E$ has split multiplicative reduction at $p$, the multiplier $(1-\alpha_p)^2$ evaluates to $0$. This corresponds to the trivial zero phenomenon of the $p$-adic $L$-function. One knows that it is a simple $0$, that is, \[(\cL^\MSD)'(E/\bQ)=\sL_p(E)\cdot\frac{L(E/\bQ,1)}{\Omega_E}\] for some $\sL_p(E)\ne 0$.
If $E$ has additive reduction at $p$, there seems no systematic way to define its $p$-adic $L$-function, but as we shall see in the next section, a twisted $p$-adic $L$-function usually does the work if the reduction is potentially ordinary.

\subsection{The cyclotomic Iwasawa Main Conjectures for elliptic curves and their twists}\label{KatoIMC}

Let $E$ be an elliptic curve with ordinary reduction at an odd prime $p$. The cyclotomic Iwasawa Main Conjecture for $E$ over $\bQ$ predicts that there is an equality\[
\Char_{\Lambda}(\Sel_{\Gr_0}(E/\bQ_\cyc)^\vee)=(\cL_E^\MSD)\text{ in }\Lambda.
\]

If one considers the $\bZ_p$-extension of $\bQ(\mu_p)$ instead, one arrives at the Iwasawa theory for the $\bZ_p^\times$-extension $\bQ(\mu_{p^\infty})$ of $\bQ$ in Kato's setting~\cite{Kato}. Then $\Gal(\bQ(\mu_{p^\infty})/\bQ)\iso \bZ_p[\Delta]$ for $\Delta=\Gal(\bQ(\mu_p)/\bQ)$ of order $p-1$ and the corresponding Iwasawa algebra is of the form $\Lambda_\Delta\defeq\bZ_p[\Delta]\llbracket T\rrbracket$. Each torsion $\Lambda_\Delta$-module $M$ naturally decomposes into isotypic components corresponding to the action of powers of the mod $p$ cyclotomic character $\omega$:\[
X=\sum_{i=0}^{p-2} X^{(i)},
\]
where $X^{(i)}$ denotes the isotypic part corresponding to $\omega^i$.

In particular, replacing the $\bQ_\infty$ by $\bQ(\mu_{p^\infty})$ in the definitions of Selmer groups, one obtains analogs of (strict) Greenberg's Selmer groups as co-finitely generated co-torsion $\Lambda_\Delta$-modules which we denote by $\Sel_{\Delta,{\Gr_0}}(E/\bQ_\cyc)$, whose $0$-th component corresponds to the Selmer group defined earlier:\[
\Sel_{\Delta,\Gr_0}(E/\bQ_\cyc)^{(0)}\iso\Sel_{\Gr_0}(E/\bQ_\cyc).\]
Moreover, the $i$-th component Selmer group agrees with strict Greenberg's Selmer group $\Sel_{\Gr_0}(E\otimes\omega^i/\bQ_\cyc)$ for the twists $E\otimes \omega^i$, i.e., one replaces $E[p^\infty]$ by $E[p^\infty]\otimes\omega^i$ in the definition of Selmer groups.

On the analytic side, Kato constructed a $p$-adic $L$-function $\cL(E,\Delta)\in\Lambda[\Delta]\llbracket T\rrbracket$ having `branches' for different powers of $\omega^{(i)}$, for $i=0,...,p-2$. In particular, its projection $\cL(E,\mathbf{1})$ to $\bZ_p\llbracket T\rrbracket$ (the trivial or $0$-th branch) agrees with $\cL_E^\MSD$ in $\Lambda$.

\begin{theorem}
	Assume that $E$ has ordinary reduction at $p\geq 3$ and that $E[p]$ is reducible. Then for each $i$, $\Char_{\Lambda}(\Sel_{\Gr_0}(E\otimes\omega^i/\bQ_\cyc)^\vee)\supset (\cL(E,\omega^i))$.
\end{theorem}\begin{proof}
This was originally due to Kato~\cite{Kato} at good ordinary primes in the non-Eisenstein case for a rational divisibility, and is extended to the integral Eisenstein case, including multiplicative reduction, by Wuthrich~\cite{Wut17}. Note that we have adopted the strict definition for Selmer groups, there there is no need to account for the trivial zero behavior (see~\cite[section 5.1]{EPW06}).
\end{proof}

The rest of the paper is devoted to proving the reversed divisibility in certain cases as well as discussing some consequences. More precisely, we will focus on the cases where $E$ is an elliptic curve having potentially multiplicative reduction at $p\geq 3$, which by~\cref{untwist} corresponds to an admissible quadratic twist (hence certain powers of $\omega$ up to some unramified twist) of an ordinary elliptic curve. We will use the methods in~\cite{GV00} to compare the Iwasawa invariants of both sides of the equation to turn the divisibility into an equality, and show that the twisted Iwasawa Main Conjecture yields the rank $0$ BSD formula for the $E$. One can further obtain the rank $1$ BSD formula for certain twist of $E$, using the anticyclotomic inputs from~\cite{KY24b}.

\subsection{The Iwasawa Main Conjectures for class groups}	
In this section, we discuss the inputs needed in~\cite{GV00} to compare the Iwasawa invariants of the Selmer groups and the corresponding $p$-adic $L$-functions. The main task is to relate the Selmer group of an elliptic curve $E$ at an Eisenstein prime to certain Selmer groups for the characters $\phi$ and $\psi$ as in~\cref{defadmissible} (which we also view as $\bZ_p^\times$-valued character via Teichm\"uller lifts, by an abuse of notation), but in the slightly more general admissible setting. 

Let $\theta$ be $\phi$ or $\psi$. For each finite set of primes $\Sigma$ containing $p$, we can define two types of Selmer groups for $\theta$, the \emph{unramified} (or \emph{strict}) Selmer group\[
\Sel_\str(\bQ_p/\bZ_p(\phi))\defeq\ker\bigl(\H^1(\bQ^\Sigma/\bQ_\cyc,\bQ_p/\bZ_p(\phi))\to \prod_{v}\frac{\H^1(\bQ_{\infty,v},\bQ_p/\bZ_p(\phi))}{\H^1_\unr(\bQ_{\infty,v},\bQ_p/\bZ_p(\phi))}\bigr),
\]
where $\H^1_\unr$ is the unramified local condition, which can be identified with\[
\ker\bigl(\H^1(\bQ^\Sigma/\bQ_\cyc,\bQ_p/\bZ_p(\phi))\to \prod_{v}\H^1(\bQ_{\infty,v},\bQ_p/\bZ_p(\phi))\bigr),
\]
 and the \emph{unramified-outside-$p$} (or \emph{relaxed}) Selmer group
\[
\Sel_\rel(\bQ_p/\bZ_p(\phi))\defeq\ker\bigl(\H^1(\bQ^\Sigma/\bQ_\cyc,\bQ_p/\bZ_p(\phi))\to \prod_{v\ne p}\H^1(\bQ_{\infty,v},\bQ_p/\bZ_p(\phi))\bigr).
\]

Their imprimitive residual versions will also be crucial to our argument:
\[
\Sel^\Sigma_\str(\F(\phi))\defeq\ker\bigl(\H^1(\bQ^\Sigma/\bQ_\cyc,\bQ_p/\bZ_p(\phi))\to \H^1(\bQ_{\infty,v},\F(\phi))\bigr),
\]
\[
\Sel^\Sigma_\rel(\F(\phi))\defeq\H^1(\bQ^\Sigma/\bQ_\cyc,\bQ_p/\bZ_p(\phi)),
\]

One knows the following form of classical Iwasawa Main Conjecture as in~\cite{GV00}.
\begin{theorem}\label{IMCoddeven}
	There exists a $p$-adic $L$-function for $\theta\ne\mathbf{1},\omega$ depending on its parity such that
	\[\begin{cases}
		\Sel_\str(\bQ_p/\bZ_p(\theta))^\vee=(\cL_\theta)& \text{ if }\theta\text{ is odd};\\
		\Sel_\rel(\bQ_p/\bZ_p(\theta))^\vee=(\cL_\theta)&\text{ if }\theta\text{ is even}.
	\end{cases}\]
\end{theorem}
\begin{proof}
	This is proved by Mazur and Wiles in~\cite{MazurWiles}
\end{proof}
In particular, one knows $\mu(\Sel_\str(\bQ_p/\bZ_p(\theta))^\vee)=0$ for odd $\theta$, and for even $\theta$ $\mu(\Sel_\rel(\bQ_p/\bZ_p(\theta))^\vee)=0$. It is immediate that $\mu(\Sel_\str(\bQ_p/\bZ_p(\theta))^\vee)=0$ for even $\theta$ as well, but $\mu(\Sel_\rel(\bQ_p/\bZ_p(\theta))^\vee)$ is known to be positive for odd $\theta$ (see~\cite[Proof of Lemma 5.9]{Greenberg1999}).

On the other hand, if $\phi$ and $\psi$ appear in the semisimplification of $E[p]$ for some ordinary elliptic curve $E$, there is a natural choice of the Selmer groups for $\phi$ and $\psi$ depending on their ramification as in~\eqref{selEtochar}. Under~\eqref{GV}, that is when the unramified character is odd and the ramified character is even, this choice agrees with the Selmer group in~\cref{IMCoddeven}. When~\eqref{GV} fails, however, the compatibility fails and Greenberg--Vatsal's method breaks down.

We conclude this section with a key lemma in describing the Selmer groups for \emph{ramified} quadratic twists of ordinary elliptic curves at Eisenstein primes.
\begin{lemma}\label[lemma]{charforE^K}
	Assume $E$ has ordinary reduction at $p$ and that $E[p]$ fits into the exact sequence of $G_\bQ$-modules\[
	0\to \F(\phi) \to E[p]\to \F(\psi) \to 0.\]
	Let $\chi$ be a ramified quadratic character. Then $\Sel(E\otimes \chi/\bQ_\cyc)$ fits into an exact sequence\[
	\Sel_?^\Sigma(\F(\phi\chi)) \to \Sel^\Sigma(E[p]\otimes \chi/\bQ_\cyc)\to \Sel_?^\Sigma(\F(\psi\chi))
	\]
	where $\Sel_?^\Sigma(\F(\theta\chi))=\Sel_\str^\Sigma(\F(\theta\chi))$ if $\theta$ is unramified, and $\Sel_?^\Sigma(\F(\theta\chi))=\Sel_\rel^\Sigma(\F(\theta\chi))$ if $\theta$ is ramified.
\end{lemma}
\begin{proof}
	This follows from the definitions and~\cref{Selinfo}(iv). Indeed, without twisting by $\chi$, it is just the definition of strict Greenberg's Selmer groups. The same holds over the quadratic field $K=K_\chi$ where $p$ ramifies, i.e.,\begin{equation}\label{Selseq/K}	\Sel_{?/K}^\Sigma(\F(\phi)) \to \Sel^\Sigma(E[p]/K_\cyc)\to \Sel_{?/K}^\Sigma(\F(\psi)),\end{equation}
	where \[\Sel_{\str/K}^\Sigma(\F(\theta))\defeq\ker\bigl(\H^1(K^\Sigma/K_\cyc,\F(\theta))\to\H^1(K_{\cyc,p},\F(\theta))\bigr),
	\]and
	\[\Sel_{\rel/K}^\Sigma(\F(\theta))\defeq\H^1(K^\Sigma/K_\cyc,\F(\theta)),
	\]
	which just follows from ordinary filtration on $E[p]$ over $G_K$ where we still denote by $p$ the unique place in $K$ above $p$.
	Now conjugation by the unique non-trivial elements in $\Gal(K_\cyc/\bQ_\cyc)$ and $\Gal(K_{\cyc,p}/\Q_{\cyc,p}) $splits~\eqref{Selseq/K} into $\pm$-eigenspaces\[
	\Sel_{?/K}^\Sigma(\F(\phi))^\pm \to \Sel^\Sigma(E[p]/K_\cyc)^\pm\to \Sel_{?/K}^\Sigma(\F(\psi))^\pm
	\]
	where $\Sel_{?/K}^\Sigma(\F(\theta))^+=\Sel_{?}^\Sigma(\F(\theta))$ and $\Sel_{?/K}^\Sigma(\F(\theta))^-=\Sel_{?}^\Sigma(\F(\theta\chi))$ from the inflation-restriction sequence. Then the residual version of~\cref{Selinfo}(iv) implies that the $-$-eigensequence is of the form\[
	\Sel_{?}^\Sigma(\F(\phi\chi))\to\Sel(E^K[p]/\bQ_\cyc)\to\Sel_{?}^\Sigma(\F(\psi\chi)),
	\]
	where $?$ only depends on the ramification of $\phi$, $\psi$ but not on that of $\phi\chi$, $\psi\chi$. This is exactly the desired sequence since $E[p^\infty]\otimes\chi$ agrees with $E^K[p^\infty]$ from the discussion in~\cref{twistvsK}.
\end{proof}

\begin{corollary}\label[corollary]{2.5.3}
	The same analysis using decomposition of Selmer groups over $K$ shows that the Selmer group $\Sel(E^K/\bQ_\cyc)$ is identified with the kernel of\[
	\H^1(\bQ^\Sigma/\bQ,E[p^\infty]\otimes\chi)\to \H^1(\bQ_{\cyc,p},\Fil^-(E[p^\infty])\otimes\chi)\times\prod_{v\ne p}\H^1(\bQ_{\cyc,v},E[p^\infty]\otimes\chi).\]
	As before, we define its imprimitive version as\[
	\Sel^\Sigma(E^K/\bQ_\cyc)=\ker\bigl(\H^1(\bQ^\Sigma/\bQ,E[p^\infty]\otimes\chi)\to \H^1(\bQ_{\cyc,p},\Fil^-(E[p^\infty])\otimes\chi)\bigr),
	\]
	with its residual version\[
	\Sel^\Sigma(E^K[p]/\bQ_\cyc)=\ker\bigl(\H^1(\bQ^\Sigma/\bQ,E[p]\otimes\chi)\to \H^1(\bQ_{\cyc,p},\Fil^-(E[p])\otimes\chi)\bigr).
	\]
\end{corollary}
We mention that the local product $\prod_{v\ne p}\H^1(\bQ_{\cyc,v},E[p^\infty]\otimes\chi)$ away from $p$ can be identified with $\prod_{v\ne p}\H^1(\bQ_{\cyc,v},E^K[p^\infty])$, but at $p$, $\Fil^-(E[p^\infty])\otimes\chi$ is not the same as $\Fil^-(E[p^\infty]\otimes\chi)$ when $\chi$ is ramified (and thus cannot be written as $\Fil^-(E^K[p^\infty])$), because the latter is undefined as $E[p^\infty]\otimes\chi$ is no longer ordinary. In particular, $\Fil^-(E[p^\infty])\otimes\chi$ is allowed to be ramified and this is where the definition for potentially ordinary reduction diverges from the ordinary case.

\section{General quadratic twists of ordinary elliptic curves and Iwasawa Main Conjectures}\label{3}
In~\cref{KatoIMC}, we discussed the twists of the Selmer group and the $p$-adic $L$-function of an elliptic curve by powers of the mod-$p$ cyclotomic character $\omega$. In this section, we discuss general ramified quadratic twists of ordinary elliptic curves at Eisenstein primes. 

Let $E$ be an elliptic curve having ordinary reduction at $p\geq 3$ such that $E[p]$ is reducible, and let $\chi$ be a quadratic character corresponding to a quadratic field $K_\chi$. If $\chi$ is unramified at $p$, then $E^{K_\chi}$ is still ordinary at $p$, and the Iwasawa Main Conjecture is known in many cases by the results of~\cite{GV00},~\cite{CGS} and~\cite{KY24}.

On the other hand, if $\chi$ is ramified at $p$, $E^{K_\chi}$ will have bad potentially ordinary reduction at $p$, and we aim to show the Iwasawa Main Conjecture in some of these cases. To go from twisting by powers of $\omega$ to general ramified quadratic twists, we just need to notice that at an odd prime, the ramification must be tame since $p$ never divides $[K_\chi:\bQ]=2$, and hence any ramified quadratic character $\chi$ decomposes into a power of $\omega$ with an unramified character locally over $\bQ_p$. In other words, if $\chi$ is a quadratic ramified character of $G_\bQ$, there exists a power $i$ (which in fact is $\frac{p-1}{2}$) of $\omega$ such that $\chi\omega^i$ is unramified at $p$. We define $\chi_\unr$ to be this unramified part which must also be quadratic so that we can write $\chi=\chi_\unr\omega^{\frac{p-1}{2}}$. Concerning potentially multiplicative reduction, we have the following.

\begin{proposition}
	Any potentially multiplicative reduction comes from a twist by $\omega^{\frac{p-1}{2}}$ of an elliptic curve of multiplicative reduction.
\end{proposition}
\begin{proof}
  Assume $E$ has potentially multiplicative reduction at $p$. From~\cref{untwist}, there is a quadratic ramified character $\chi$ corresponding to a quadratic field $K_\chi$ such that $E=E_0^{K_\chi}$ for some $E_0$ of multiplicative reduction. Then $E$ is the $\omega^{\frac{p-1}{2}}$-twist of $E_0^{K_\unr}$ where $K_\unr$ is the quadratic field corresponding to $\chi_\unr$ and $E_0^{K_\unr}$ has multiplicative reduction.
\end{proof}

\begin{remark}
	It is easy to see that rewriting $E=E_0^{K_\chi}$ as $E_0^{K_\unr}\otimes \omega^{\frac{p-1}{2}}$ won't affect admissibility, since it is a property about the twisted object.
\end{remark}

Thus without loss of generality, we are free to replace any ramified quadratic twist by an $\omega^{\frac{p-1}{2}}$-twist. In particular, Kato's $p$-adic $L$-function and the divisibility on Iwasawa Main Conjecture covers all ramified quadratic twists, including all those corresponding to elliptic curves with potentially multiplicative reduction. We denote the $\chi$-twisted $p$-adic $L$-function by $\cL(E,\chi)$. Then Kato's divisibility reads\begin{equation}\tag{Kato}\label{Katogeneral}
\Char_{\Lambda}(\Sel_{\Gr_0}(E\otimes\chi/\bQ_\cyc)^\vee)\supset (\cL(E,\chi)).
\end{equation}

\section{Comparing Iwasawa invariants for admissible quadratic twists}\label{4}
In this section, we extend Greenberg--Vatsal's method to compare Iwasawa invariants for admissible quadratic twists of elliptic curves in the Eisenstein case. Let $E$ be an elliptic curve of ordinary reduction at $p\geq 3$ such that $E[p]$ fits into the exact sequence of $G_\bQ$-modules\[
0\to \F(\phi) \to E[p]\to \F(\psi) \to 0
\]
with $\phi\psi=\omega$. For a quadratic character $\chi$, $E[p]\otimes \chi$ fits into the short exact sequence\[
0\to \F(\phi\chi) \to E[p]\otimes\chi \to \F(\psi\chi) \to 0
\]
with $(\phi\chi)(\psi\chi)=\omega$, so in particular, one of $\phi\chi,\psi\chi$ is odd and the other is even. Unlike for $\phi$ and $\psi$, $\phi\chi$ and $\psi\chi$ can now both be ramified.

We now record a lemma allowing us to compute the $\mu$- and $\lambda$-invariant of the strict Greenberg's Selmer group for quadratic twists of $E$ using knowledge about Selmer groups of $\phi$ and $\psi$. These are analogous to those in~\cite{GV00} for an ordinary elliptic curve. Since we are mostly interested in elliptic curves with potentially ordinary reduction, we will assume that $\chi$ is an \emph{ramified} quadratic character.

\begin{lemma}\label[lemma]{GV2.4}
	(i) The global-to-local map\[
	\H^1(\bQ^\Sigma/\bQ_\cyc,E[p^\infty]\otimes\chi)\to \H^1(\bQ_{\cyc,p},\Fil^-(E[p^\infty])\otimes\chi)\times\prod_{v\ne p}\H^1(\bQ_{\cyc,v},E[p^\infty]\otimes\chi)
	\] defining $\Sel(E^{K_\chi}/\bQ_\cyc)$ is surjective. Consequently,\[\Sel^\Sigma(E^{K_\chi}/\bQ_\cyc)/\Sel(E^{K_\chi}/\bQ_\cyc)\iso\prod_{v\in\Sigma-\{p\}}\H^1(\bQ_{\cyc,v},E[p^\infty]\otimes\chi).\]\\
	(ii) For a prime $\ell\ne v,\infty$, $\mu(\H^1(\bQ_{\cyc,\ell},E[p^\infty]\otimes\chi)^\vee)=0$ and $\lambda(\H^1(\bQ_{\cyc,\ell},E[p^\infty]\otimes\chi)^\vee)=s_\ell d_\ell$ for $s_\ell,d_\ell$ described in~\cite[Propositoin (2.4)]{GV00}.\\
	(iii) $\mu(\Sel(E^{K_\chi}/\bQ_\cyc))=0\Leftrightarrow \mu(\Sel^\Sigma(E^{K_\chi}[p]/\bQ_\cyc))=0.$
\end{lemma}
\begin{proof}
	(i) First, note that $\Sel(E^{K_\chi}/\bQ_\cyc)$ is $\Lambda$-cotorsion, which follows from the cotorsionness after base change to $K_\chi$ by results of Kato and Rohrlich (see~\cite[Theorem 1.5]{Greenberg1999}). Second, $\H^0(\bQ_\cyc, E[p^\infty]\otimes\chi)=\H^0(\bQ_\cyc,E^{K_\chi}[p^\infty])$ is finite by the theorem of Ribet~\cite{Ribet}. Third, one can rewrite the local term $\H^1(\bQ_{\cyc,p},\Fil^-(E[p^\infty])\otimes\chi)$ as $\H^1(\bQ_{\cyc,p}.E[p^\infty]\otimes\chi)/\H^1_\ord$ where the local condition $\H^1_\ord$ can be expressed as $\im(\H^1(\bQ_{\cyc,p},\Fil^+(E[p^\infty])\otimes\chi)\to\H^1(\bQ_{\cyc,p}.E[p^\infty]\otimes\chi))$, which has kernel bounded by $\#\H^0(\bQ_{\cyc,p},\Fil^-(E[p^\infty])\otimes\chi)<\infty$ as $\chi$ is ramified. Now from~\cite[Lemma 2.3]{Greenberg1999}, the $\Lambda$-corank of the domain (hence the image) is $1$. Note that $\Fil^+(E[p^\infty])\otimes\chi$ can be neither $\bQ_p/\bZ_p$ nor $\bQ_p/\bZ_p(1)$ since it is ramified but becomes unramified mod $p$ when $\chi$ is ramified and quadratic. Now the surjectivity follows from~\cite[Proposition A.2]{PollackWeston2011}. The comparison between the imprimitive and primitive Selmer group follows.\\
	(ii) This follows exactly from~\cite[Proposition (2.4)]{GV00} for $E^{K_\chi}$ since it is independent from reduction type at $p$.\\
	(iii) Regarding (i), (ii) above, this follows from the proof of~\cite[Proposition 3.1.3]{LY26}. 
	
\end{proof}

\subsection{The admissible twists}
For admissible twists, we will directly adopt the computation in~\cite{GV00} to compute the $\mu$- and $\lambda$-invariants of the Selmer groups to the $p$-adic $L$-functions. Let $\chi$ be a quadratic ramified character with opposite parity to the unramified character of $\phi$, $\psi$, and let $K_\chi$ be the quadratic field corresponding to $\chi$. We will divide the admissibility into two subcategories according to the parity of the quadratic twist.

\subsubsection{The $+$-admissible twists}

Under~\eqref{GV}, the unramified character is odd, so $\chi$ must be even. Then there are two cases:
\begin{enumerate}
	\item $\phi$ is ramified and even, and $\psi$ is unramified and odd;
	\item $\phi$ is unramified and odd, and $\psi$ is ramified and even.
\end{enumerate}

In case (i), the imprimitive residual Selmer group for $E^{K_\chi}$ fits into the following exact sequence by~\cref{charforE^K}:\[
\Sel_\rel^\Sigma(\F(\phi\chi)) \to \Sel^\Sigma(E[p]\otimes \chi/\bQ_\cyc)\to \Sel_\str^\Sigma(\F(\psi\chi)).\]
The vanishing of $\mu$-invariant of $\Sel^\Sigma(E[p]\otimes \chi/\bQ_\cyc)^\vee$ immediately follows from the vanishing of $\mu$ for $\Sel_\rel^\Sigma(\F(\phi\chi))^\vee$ (as $\phi\chi$ is even) and $\Sel_\str^\Sigma(\F(\psi\chi))^\vee$. Similarly, the vanishing of $\mu$ also follows in case (ii). 

Now note that the $\lambda$-invariant is constant in the isogeny class of $E^{K_\chi}$ (which replies on the fact that composing with dual isogeny yields a multiplication-by-degree map, and is independent of reduction type), so to compute the $\lambda$-invariant of $\Sel^\Sigma(E[p]\otimes \chi/\bQ_\cyc)$, we can assume we are in case (i). In particular, $\phi\chi$ is even, and $\psi\chi$ is odd and ramified.

Then we have a commutative diagram{\small\begin{equation*}\begin{tikzcd}
		& \Sel_{\rel}^\Sigma(\F(\phi\chi)) \ar[r]\ar[d]&\Sel^\Sigma(E[p]\otimes\chi/\bQ_\cyc)  \ar[r]\ar[d]&  \Sel^\Sigma_\str(\F(\psi\chi)) \ar[d] \\
		&  \H^1(\bQ^\Sigma/\bQ_\cyc,\F(\phi\chi))\ar[r]\ar[d]& \H^1(\bQ^\Sigma/\bQ_\cyc,E[p]\otimes\chi)\ar[r]\ar[d]& \H^1(\bQ^\Sigma/\bQ_\cyc,\F(\psi\chi))
		\ar[d]\\
		& 0 \ar[r]& \H^1(\bQ_{\infty,p},\Fil^-(E[p])\otimes\chi)\ar[r,"="]& \H^1(\bQ_{\infty,p},\F(\psi\chi)),
\end{tikzcd}\end{equation*}}
where the second row is also exact on the left and right since $\phi\chi\ne\omega$ and $\psi\chi\ne\mathbf{1}$. Then the snake lemma yields the short exact sequence\[
0\to\Sel_{\rel}^\Sigma(\F(\phi\chi))\to \Sel^\Sigma(E[p]\otimes\chi/\bQ_\cyc) \to \Sel^\Sigma_\str(\F(\psi\chi))\to 0,
\]
from which we obtain\begin{itemize}
	\item $\mu(\Sel^\Sigma(E[p]\otimes\chi/\bQ_\cyc)^\vee)=0$,
	\item $\lambda(\Sel^\Sigma(E\otimes\chi/\bQ_\cyc)^\vee)=\lambda(\Sel_\rel^\Sigma(\bQ_p/\bZ_p(\phi\chi))^\vee)+\lambda(\Sel_\str^\Sigma(\bQ_p/\bZ_p(\psi\chi))^\vee)$,
\end{itemize}
provided that $\phi\chi\ne \mathbf{1}$ so that $\H^0(\bQ,\F(\phi\chi))=\H^0(\bQ,E[p]\otimes\chi)=0$. The second equation is then from~\cite[Proposition 2.8]{GV00}, where we use the replacement $D^{I_\eta}=0$ in Remark 2.9 rather than assuming $D$ is unramified in \emph{op. cit.} The assumption $\phi\chi\ne \mathbf{1}$ is reasonable because Iwasawa Main Conjecture is inconclusive for the trivial character.

On the analytic side, again assume that $\phi\chi\ne \mathbf{1}$ so that the Kubota--Leopoldt $p$-adic $L$-function is holomorphic. Then~\cite[Theorem (3.11)]{GV00} implies that $\mu(\cL_p^\Sigma(E,\chi))$ is congruent mod imprimitive $p$ to their $p$-adic $L$-functions associated to $\phi\chi$ and $\psi\chi$. In particular, we obtain\begin{itemize}
	\item $\mu(\cL^\Sigma(E,\chi))=0$,
	\item $\lambda(\cL^\Sigma(E,\chi))=\lambda(\cL^\Sigma_{\phi\chi})+\lambda(\cL^\Sigma_{\psi\chi})$.
\end{itemize}
Then as in~\cite{GV00}, Mazur--Wiles's theorem shows that $\lambda(\Sel_\rel^\Sigma(\bQ_p/\bZ_p(\phi\chi))^\vee)=\lambda(\cL^\Sigma_{\phi\chi})$ and $\lambda(\Sel_\rel^\Sigma(\bQ_p/\bZ_p(\psi\chi))^\vee)=\lambda(\cL^\Sigma_{\psi\chi})$, so in particular $\lambda(\Sel^\Sigma(E\otimes\chi/\bQ_\cyc)^\vee)=\lambda(\cL^\Sigma(E,\chi))$. Thus as before,~\cite[Proposition (2.4)]{GV00} applied to $E^{K_\chi}$ implies that $\lambda(\Sel(E\otimes\chi/\bQ_\cyc)^\vee)=\lambda(\cL(E,\chi))$, which, combined with the vanishing of $\mu$-invariants (see~\cref{GV2.4}(iii)) and Kato's divisibility~\eqref{Katogeneral}, proves the Iwasawa Main Conjecture for the twist $E^{K_\chi}$, i.e.,\[
\Char_{\Lambda}(\Sel(E^{K_\chi}/\bQ_\cyc)^\vee)=(\cL(E,\chi)).\]

\subsubsection{The $-$-admissible twists}
In addition to the condition~\eqref{GV} with twists by an even $\chi$, the theory developed in \emph{loc. cit.} in fact covers another type of admissible twists, namely the opposite condition of~\eqref{GV} with an odd twist. More specifically, we assume the unramified character of $\phi,\psi$ is even, and $\chi$ be an odd ramified quadratic character. Then as before, there are two cases:
\begin{enumerate}
	\item $\phi$ is ramified and odd, and $\psi$ is unramified and even;
	\item $\phi$ is unramified and even, and $\psi$ is ramified and odd.
\end{enumerate}
In case (i), there is an exact sequence coming from~\cref{charforE^K}:\[
\Sel_\rel^\Sigma(\F(\phi\chi)) \to \Sel^\Sigma(E[p]\otimes \chi/\bQ_\cyc)\to \Sel_\str^\Sigma(\F(\psi\chi)),
\]
and as before, the vanishing of $\mu(\Sel^\Sigma(E[p]\otimes\chi/\bQ_\cyc)^\vee)$ follows from the vanishing of $\mu(\Sel_\rel^\Sigma(\F(\phi\chi))^\vee)$ (as $\phi\chi$ is even) and $\mu(\Sel_\str^\Sigma(\F(\psi\chi))^\vee)$. The vanishing of $\mu$ in case (ii) follows similarly.

Again, to compute $\lambda(\Sel^\Sigma(E[p]\otimes\chi/\bQ_\cyc)^\vee)$, we may assume we are in case (i). In this case, $\phi\chi$ is even, and $\psi\chi$ is ramified and odd. We still have the same diagram{\small\begin{equation*}\begin{tikzcd}
			& \Sel_{\rel}^\Sigma(\F(\phi\chi)) \ar[r]\ar[d]&\Sel^\Sigma(E[p]\otimes\chi/\bQ_\cyc)  \ar[r]\ar[d]&  \Sel^\Sigma_\str(\F(\psi\chi)) \ar[d] \\
			&  \H^1(\bQ^\Sigma/\bQ_\cyc,\F(\phi\chi))\ar[r]\ar[d]& \H^1(\bQ^\Sigma/\bQ_\cyc,E[p]\otimes\chi)\ar[r]\ar[d]& \H^1(\bQ^\Sigma/\bQ_\cyc,\F(\psi\chi))
			\ar[d]\\
			& 0 \ar[r]& \H^1(\bQ_{\infty,p},\Fil^-(E[p])\otimes\chi)\ar[r,"="]& \H^1(\bQ_{\infty,p},\F(\psi\chi)).
\end{tikzcd}\end{equation*}}
Now as in the case of $+$-admissible twists, we obtain\begin{itemize}
	\item $\mu(\Sel^\Sigma(E[p]\otimes\chi/\bQ_\cyc)^\vee)=0$,
	\item $\lambda(\Sel^\Sigma(E\otimes\chi/\bQ_\cyc)^\vee)=\lambda(\Sel_\rel^\Sigma(\bQ_p/\bZ_p(\phi\chi))^\vee)+\lambda(\Sel_\str^\Sigma(\bQ_p/\bZ_p(\psi\chi))^\vee)$,
\end{itemize}
provided that $\phi\chi\ne \mathbf{1}$. 
On the analytic side, the proof of~\cite[Theorem (3.11)]{GV00} and~\cite[Theorem (2.10)]{Vat99} still applies because the admissible sign is now $-$. Thus we obtain the congruence\[
\cL^\Sigma(E,\chi)=\cL^\Sigma_{\phi\chi}\cL^\Sigma_{\psi\chi}\pmod{p\Lambda},
\]
as well as the relations\begin{itemize}
	\item $\mu(\cL^\Sigma(E,\chi))=0$,
	\item $\lambda(\cL^\Sigma(E,\chi))=\lambda(\cL^\Sigma_{\phi\chi})+\lambda(\cL^\Sigma_{\psi\chi})$.
\end{itemize}
Again, Mazur--Wiles's theorem allow us to equate the algebraic Iwasawa invariants of $E$ to the analytic Iwasawa invariants, which combined with Kato's divisibility yields the Iwasawa Main Conjecture. We have thus proved the following.
\begin{theorem}\label{IMC}
	Let $E$ be an elliptic curve with ordinary reduction at $p$ such that $E[p]$ is reducible over $G_\bQ$. Let $\chi$ be a ramified quadratic admissible twist of $E$. Assume $\H^0(\bQ,E'[p]\otimes\chi)=0$ for all $E'$ isogenous to $E$. Then the Iwasawa Main Conjecture for the twist $E^{K_\chi}$ is true, i.e.,\[
	\Char_{\Lambda}(\Sel(E^{K_\chi}/\bQ_\cyc)^\vee)=(\cL(E,\chi)).\]
\end{theorem}
Here the non-torsion hypothesis just means $\F(\mathbf{1})$ is not a $G_\bQ$-invariant submodule or quotient of $E[p]\otimes\chi$. For similar results on many unramified admissible twists (but only the $+$-cases), see~\cite{GV00},~\cite{CGS} and~\cite{KY24}.

\subsection{The non-admissible twists}\label{nonadGr}
In the non-admissible situation, we shall see that the conditions\begin{itemize}
	\item $\mu(\Sel^\Sigma(E[p]\otimes\chi/\bQ_\cyc)^\vee)=0$,
	\item $\lambda(\Sel^\Sigma(E\otimes\chi/\bQ_\cyc)^\vee)=\lambda(\Sel_?(\phi\chi/\bQ_\cyc)^\vee)+\lambda(\Sel_?(\psi\chi/\bQ_\cyc)^\vee)$,
\end{itemize}   and the congruence \begin{itemize}
\item $\cL_p(E,\chi)\equiv \cL_{\phi\chi}\cL_{\psi\chi}\pmod{p\Lambda}$
\end{itemize} are no longer expected to hold, due to the misalignment in the choices of Selmer groups for $\phi,\psi$ using ramification v.s. using parity. For example, when $\phi$ is even and ramified, but $\chi$ is odd, then we are in the situation of the following diagram{\small\begin{equation*}\begin{tikzcd}
		& \Sel_{\rel}^\Sigma(\F(\phi\chi)) \ar[r]\ar[d]&\Sel^\Sigma(E[p]\otimes\chi/\bQ_\cyc)  \ar[r]\ar[d]&  \Sel^\Sigma_\str(\F(\psi\chi)) \ar[d] \\
		&  \H^1(\bQ^\Sigma/\bQ_\cyc,\F(\phi\chi))\ar[r]\ar[d]& \H^1(\bQ^\Sigma/\bQ_\cyc,E[p]\otimes\chi)\ar[r]\ar[d]& \H^1(\bQ^\Sigma/\bQ_\cyc,\F(\psi\chi))
		\ar[d]\\
		& 0 \ar[r]& \H^1(\bQ_{\infty,p},\Fil^-(E[p])\otimes\chi)\ar[r,"="]& \H^1(\bQ_{\infty,p},\F(\psi\chi)),
\end{tikzcd}\end{equation*}}
where $\phi\chi$ is odd. If $\phi\chi$ also happens to be ramified, the $\Sel_\rel^\Sigma(\F(\phi\chi))^\vee$ will have a positive $\mu$-invariant, hence $\Sel^\Sigma(E[p]\otimes\chi/\bQ_\cyc)^\vee$ will, too. This inadmissibility introduces at least three major issues:\begin{itemize}
	\item With positive $\mu$-invariants, the lambda invariants no longer additive in short exact sequences.
	\item Even worse, $\Sel^\Sigma_\rel(\F(\phi\chi))$ is not the correct Selmer group appearing in the Iwasawa Main Conjecture for an odd character $\phi\chi$, so Mzaur--Wiles's theorem cannot be applied. Same with $\Sel^\Sigma_\str(\F(\phi\chi))$ for an even character $\psi\chi$.
	\item On the analytic side, the period with the prescribed sign becomes non-admissible, and the congruence fails completely for deep reasons.
	\end{itemize}

Therefore, the methods of Greenberg--Vatsal break down in the non-admissible case, and it probably requires completely new ideas to deal with it. If $E$ has good ordinary reduction, Castella--Grossi--Skinner~\cite{CGS} take a different path and use Beilinson--Flach elements to connect the cyclotomic Iwasawa Main Conjectures to anticyclotomic Iwasawa Main Conjectures, which successfully bypass the difficulty in the Iwasawa invariants comparison. Their approach has been generalized in~\cite{KY24} for more ordinary Eisenstein primes, but it is unclear at the moment if it also covers (ramified) twists of ordinary elliptic curves. Therefore, we will not attempt to treat the non-admissible Iwasawa Main Conjecture in this paper. Instead, we verify Greenberg's $\mu=0$ conjecture, another major conjecture in Iwasawa theory for elliptic curves.

\begin{conjecture}[Greenberg]\label{Greenberg}
	Let $E$ be an elliptic curve and $p$ be a prime. Assume that $\Sel_{\Gr_0}(E/\bQ_\infty)$ is $\Lambda$-cotorsion. Then there exists an elliptic curve $E'$ in the $\bQ$-isogeny class of $E$ such that $\mu(\Sel_{\Gr_0}(E'/\bQ_\infty)^\vee)=0$.
\end{conjecture}
In the good ordinary Eisenstein case, this was proved in~\cite{GV00} under their assumption~\eqref{GV}. No substantial progress has been made in other cases until recently, when the non-\eqref{GV} case was proved in~\cite{LY26}, which also covers multiplicative reduction. The irreducible case, on the other hand, remains widely open, but the approach in~\cite{LY26} does allow irreducibility conditional on deep inputs, and in very special cases for $p=2$, the irreducible $\mu=0$ can be confirmed.

In the remaining of this section, we study Greenberg's conjecture for ramified quadratic twists, i.e., corresponding to elliptic curves of potentially ordinary reductions, which makes sense since the Selmer group is still known to be $\Lambda$-cotorsion (see the proof of~\cref{GV2.4}(i)).

For admissible twists, we showed in the previous subsection that the $\mu$-invariant is always $0$, without the need to move to an isogenous curve. The example at the beginning of this subsection shows that it is not true for non-admissible twists. We will closely follow the methods in~\cite{LY26} to treat the non-admissible case.

First, notice that the only source of a positive $\mu$-invariant in the Eisenstein case is those coming from $\Sel^\Sigma_\rel(\F(\theta))$ the relaxed Selmer group for odd and ramified characters. If such Selmer group becomes a subgroup of $\Sel^\Sigma(E[p]\otimes\chi/\bQ_\cyc)$, the latter is bound to have positive $\mu$-invariant as well. Thus when it happens, it is necessary to move to an isogenous curve where such $\Sel^\Sigma_\rel(\F(\theta))$ is a quotient. This is possible due to the following reasons:\begin{itemize}
	\item Since $\phi\psi=\omega$ and $\chi$ is quadratic, only one of $\phi\chi$ and $\psi\chi$ is odd;
	\item Ribet's lemma guarantees that one can find a non-split extension\[
	0\to\F(\phi)\to E[p]\to\F(\omega\phi^{-1})\to 0\]
	which remains non-split after twist:\begin{equation}\label{nonsplit}
	0\to \F(\phi\chi)\to E[p]\otimes\chi\to\F(\omega\phi^{-1}\chi)\to 0,\end{equation} where $\phi$ is any of the two characters in the semisimplification of $E[p]$.
\end{itemize}
Thus we may choose the odd ramified character (when exists) in $\phi\chi,\psi\chi$ to be a non-split quotient of $E[p]\otimes\chi$. Without loss of generality assume $\psi\chi$ is odd and ramified. In order to see a positive $\mu$-invariant, the associated Selmer group must be relaxed, meaning $\psi$ is ramified (matching the ordinary filtration of $E[p]$), and the ramification of $\psi$ and $\chi$ shall not cancel. To get a non-admissible twist, $\chi$ and $\psi$ must then be of the opposite parity, so either\begin{itemize}
	\item $\psi$ is ramified and even, $\chi$ is ramified and odd; or
	\item $\psi$ is ramified and odd, $\chi$ is ramified and even.
\end{itemize}
In either case, we shall see a diagram{\small\begin{equation*}\begin{tikzcd}
			& \Sel_{\str}^\Sigma(\F(\phi\chi)) \ar[r]\ar[d]&\Sel^\Sigma(E[p]\otimes\chi/\bQ_\cyc)  \ar[r]\ar[d]&  \Sel^\Sigma_\rel(\F(\psi\chi)) \ar[d] \\
			&  \H^1(\bQ^\Sigma/\bQ_\cyc,\F(\phi\chi))\ar[r]\ar[d]& \H^1(\bQ^\Sigma/\bQ_\cyc,E[p]\otimes\chi)\ar[r]\ar[d]& \H^1(\bQ^\Sigma/\bQ_\cyc,\F(\psi\chi))
			\ar[d]\\
			& \H^1(\bQ_{\infty,p},\F(\phi\chi)) \ar[r,"="]& \H^1(\bQ_{\infty,p},\Fil^-(E[p])\otimes\chi)\ar[r]& 0,
\end{tikzcd}\end{equation*}}
where $\phi\chi$ is ramified and even, and $\psi\chi$ is ramified and odd. Again, according to~\cref{GV2.4}(iii), it suffices to show $\Sel^\Sigma(E[p]\otimes\chi/\bQ_\cyc)$ as above has vanishing $\mu$-invariant.

\begin{proof}[Proof of Greenberg's conjecture for quadratic twists of ordinary curves:]
$ $\newline As in~\cite{LY26}, we shall consider the following diagram (enlarging $\Sigma$ if necessary)\begin{center}{\small\begin{equation*}\begin{tikzcd}	
				\H^1(\bQ^\Sigma/\bQ_\cyc,E[p]\otimes\chi) \ar[r,"\res_{L/\bQ}"]\ar[d,"loc_\bQ"]& \H^1(L^\Sigma/L_\cyc,E[p]\otimes\chi)\ar[d,"loc_{L}"]\ar[r,"\proj_L"] &\H^1(L^\Sigma/L_\cyc,\F(\phi\chi))\ar[d,"\loc_{L,\F(\phi\chi)}"]\\
				\H^1(\bQ_{\cyc,p},E[p]\otimes\chi)\ar[r,"\prod_{w_i\mid p}\res"]&\displaystyle\prod_{w_i\mid p}\H^1(L_{\cyc,w_i},E[p]\otimes\chi)\ar[r,"\prod \proj_i"]&\displaystyle\prod_{w_i\mid p}\H^1(L_{\cyc,w_i},\F(\phi\chi)),\\	 	\end{tikzcd}
	\end{equation*}}
\end{center}
where $L=\bQ(E[p])K_\chi$, and $\proj$ are the projection maps induced by the ordinary projection $E[p]\to E[p]/\Fil^+(E[p])\iso\F(\phi)$ with a twist by $\chi$. Note that since these cohomology groups have finite coefficients, $\mu=0$ is equivalent to being finite. To show $\mu(\Sel^\Sigma(E[p]\otimes\chi/\bQ_\cyc))=0$, we need to show no global $\mu$ from $\H^1(\bQ^\Sigma/\bQ_\cyc,E[p]\otimes\chi)$ under $\loc_\bQ$ is fully absorbed in $\H^1(\bQ_{\cyc,p},\Fil^+(E[p])\otimes\chi)=\H^1(\bQ_{\cyc,p},\F(\psi\chi))$. We will show a sufficient condition for this, that is no global $\mu$ from $\H^1(\bQ^\Sigma/\bQ_\cyc,E[p]\otimes\chi)$ under $\prod_{w_i\mid p}\res\circ \loc_\bQ$ is absorbed in $\prod_{w_i\mid p}\H^1(L_{\cyc,w_i},\F(\psi\chi))$, which is equivalent to the finiteness of $\ker(\prod\proj_i\circ\prod_{w_i\mid p}\res\circ\loc_\bQ)$.

From the commutativity of the diagram, it suffices to show $\ker(\loc_{L,\F(\phi\chi)}\circ\proj_L\circ\res_{L/\bQ})$ is finite, which we show in two steps:\begin{enumerate}
	\item $\ker(\proj_L\circ\res_{L/\bQ})$ is finite;
	\item $\ker(\loc_{L,\F(\phi\chi)})$ is finite.
	\end{enumerate}
\noindent\emph{Proof of step (i): }This follows from the following two facts: (1) $\F(\psi\chi)$ is not a $G_\bQ$-submodule of $E[p]\otimes\chi$, i.e., the non-splitness of~\eqref{nonsplit}, and (2) $\im(\res_{L/\bQ})$ lies in the invariant subspace $\H^1(L^\Sigma/L_\cyc,E[p]\otimes\chi)^{L_\cyc/\bQ_\cyc}$. One can still idenfity $\Gal(L_\cyc/\bQ_\cyc)$ with $\Gal(L/\bQ)$ for degree reasons as in~\cite[section 3.1.3]{LY26} by noting that $\bQ(E[p])K_\chi$ remains non-abelian over $\bQ$ if $\bQ(E[p])/\bQ$ is. Now the proof in \emph{loc. cit.} shows that the above two facts force the kernel in (i) to be finite.\\
\noindent\emph{Proof of step (ii): }This follows from the $p$-extension principle applied to the field extension $\bQ(E[p])K_\chi/\bQ(\phi,\psi)K_\chi$ which has degree $1$ or $p$, together with the Ferrero--Washington theorem applied to the abelian extension $\bQ(\phi,\psi)K_\chi=\bQ(\phi,\psi,\chi)$ of $\bQ$. See again \emph{loc. cit.} for more details.
\end{proof}

We can say a bit more about the distribution of $\mu$-invariants in an isogeny class. For example, no tamely ramified and even character exists if $p=3$, so we do not need to consider the exceptional case where one must move to an isogenous curve with reduced $\mu$-invariant, and all quadratic twists readily have vanishing $\mu$-invariant.

\section{Cyclotomic control theorem for quadratic twists}\label{cyccontrol}
In this section, we extend Greenberg's computation in~\cite[section 4]{Greenberg1999} to cover twists of ordinary primes in the rank $0$ case. Recall that $\Lambda=\llbracket \Gamma\rrbracket\iso\bZ_p\llbracket T\rrbracket$ is the Iwasawa algebra with $\Gamma=\Gal(\bQ_\cyc/\bQ)$. Let $\Sigma$ be a finite set of primes containing $p$, $\infty$, and those there $E$ or $E^{K_\chi}$ has bad reduction.

We begin with a lemma in~\cite{Greenberg1999}.\begin{lemma}(Control theorem)
	Assume $E/\bQ$ has potentially ordinary reduction at $p\geq 3$, then the restriction map\[
	\Sel_E(\bQ)_p\to \Sel_E(\bQ_\infty)^\Gamma_p
	\]
	has finite kernel and cokernel.
\end{lemma}
\begin{proof}
	This is in the discussion after the proof of~\cite[Prososition 3.8]{Greenberg1999}.
\end{proof}
In our setting, this lemma applies to a ramified quadratic twist of an ordinary elliptic curve $E$, where the map reads\[
\Sel_{\Gr_0}(E\otimes\chi/\bQ)\to \Sel_{\Gr_0}(E\otimes\chi/\bQ_\infty)^\Gamma,
\]
where over $\bQ$, the Selmer groups $\Sel_{\Gr_0}(E/\bQ)$ are defined as the kernel of{\small\[
\H^1(\bQ^\Sigma/\bQ, E[p^\infty]) \to\frac{\H^1(\bQ_p,E[p^\infty])}{\im(\H^1(\bQ_p,\Fil^+(E[p^\infty]))\to\H^1(\bQ_p,E[p^\infty]))_\div}  \times\prod_{v\in\Sigma-\{p\} }\H^1(\bQ_{\cyc,v},E[p^\infty]),\]}
with $\Sel_{\Gr_0}(E\otimes\chi/\bQ)$ defined as the kernel of{\small\begin{align*}
	\H^1(\bQ^\Sigma/\bQ, E[p^\infty]\otimes\chi) \to&\frac{\H^1(\bQ_p,E[p^\infty]\otimes\chi)}{\im(\H^1(\bQ_p,\Fil^+(E[p^\infty])\otimes\chi)\to\H^1(\bQ_p,E[p^\infty]\otimes\chi))_\div} \\
	 &\times\prod_{v\in\Sigma-\{p\}}\H^1(\bQ_{\cyc,v},E[p^\infty]\otimes\chi),\end{align*}}
 where `$\div$' denotes the divisible part. That these agree with the classical Selmer group coming from Kummer image is in~\cite[section 2]{Greenberg1999} for both $\bQ$ and $K_\chi$. Note that at the infinite level taking the divisible part is not needed because the images are already divisible. 
 
 As the Selmer groups for $E\otimes \chi$ are identified with the Selmer groups for the quadratic twist $E^{K_\chi}$, which has potentially ordinary reduction, so the control theorem applies.

Thus, under the assumption that $\Sel_{\Gr_0}(E\otimes\chi/\bQ)$ is finite, $\Sel_{\Gr_0}(E\otimes\chi/\bQ_\cyc)^\Gamma$ will be finite, and its Pontryagin dual $\Sel_{\Gr_0}(E\otimes\chi/\bQ_\cyc)^\vee/T$ will also be finite, with $\Sel_{\Gr_0}(E\otimes\chi/\bQ_\cyc)^\vee$ generated by an element $f_{E\otimes\chi}(T)\in\Lambda$ with $T\nmid f_{E\otimes\chi}(T)$. The aim of this section is to compute its constant term $f_{E\otimes\chi}(0)$ up to $p$-adic units.

\begin{theorem}\label{fE(0)}
Assume $E/\bQ$ has ordinary reduction at $p\geq 3$ and that $\chi$ is a quadratic ramified character. If $\Sel_{\Gr_0}(E\otimes\chi/\bQ)$ is finite, then\[
f_{E\otimes\chi}(0)\sim\bigl(\prod_{p\ne v\mid N(E^{K_\chi})}\ord_p(c_{v,E^{K_\chi}})\bigr)\vert\Sel(E^{K_\chi}/\bQ)\vert/\vert E^{K_\chi}(\bQ)[p^\infty]\vert^2.
\]
\end{theorem}

This formula will be proved in a series of lemmas and propositions. We write $S_{E\otimes\chi}=\Sel_{\Gr_0}(E\otimes\chi/\bQ_\cyc)$ and $X_{E\otimes\chi}=(S_{E\otimes\chi})^\vee$, then according to~\cite[Lemma 4.2]{Greenberg1999}, under the assumptions of the theorem, \[f_{E\otimes\chi}(0)\sim\vert S_{E\otimes\chi}^\Gamma\vert/\vert (S_{E\otimes\chi})_\Gamma\vert.\]

Consider the diagram\\
{\small\begin{tikzcd}
	0\ar[r] & \Sel_{\Gr_0}(E\otimes\chi/\bQ)\ar[r]\ar[d,"s"]& \H^1(\bQ^\Sigma/\bQ,E[p^\infty]\otimes\chi) \ar[r,"\loc_S"]\ar[d,"h"]& \mathcal{G}\ar[r]\ar[d,"g"]&0 \\
	0\ar[r]& \Sel_{\Gr_0}(E\otimes\chi/\bQ_\cyc)^\Gamma \ar[r]& \H^1(\bQ^\Sigma/\bQ_\infty,E[p^\infty]\otimes\chi)^\Gamma  \ar[r]&\mathcal{G}(\bQ_\cyc)^\Gamma&
\end{tikzcd}}\\
where{\small\[\mathcal{G}=\im\{\H^1(\bQ^\Sigma/\bQ,E[p^\infty]\otimes\chi)\xrightarrow{\loc_\bQ}\frac{\H^1(\bQ_p,E[p^\infty]\otimes\chi)}{\H^1(\bQ_p,\Fil^-(E[p^\infty])\otimes\chi)}\times\prod_{v\in\Sigma-\{p\}}\H^1(\bQ_v,E[p^\infty]\otimes\chi),\]}
and
{\small\begin{align*}\mathcal{G}(\bQ_\cyc)=\im\{\H^1(\bQ^\Sigma/\bQ_\cyc,E[p^\infty]\otimes\chi)\xrightarrow{\loc_{\bQ_\cyc}}&\frac{\H^1(\bQ_{\cyc,p},E[p^\infty]\otimes\chi)}{\H^1(\bQ_{\cyc,p},\Fil^-(E[p^\infty])\otimes\chi)}\\&\times\prod_{v\in\Sigma-\{p\}}\H^1(\bQ_{\cyc,v},E[p^\infty]\otimes\chi),\end{align*}}
are the global-to-local images at the base level and the $\bQ_\cyc$-level respectively.

Thus we have $\vert S_{E\otimes\chi}^\Gamma\vert=\vert \Sel_{\Gr_0}(E\otimes\chi/\bQ)\vert\vert\coker(s)\vert/\vert\ker(s)\vert$, where the snake lemma identifies the quotient $\vert\coker(s)\vert/\vert\ker(s)\vert$ with $\vert\ker(g)\vert/\vert\ker(h)\vert$ as $\coker(h)=0$, which follows from the cohomological dimension of $\Gamma$ being $1$. Here we must explain why $\ker(h)$ is finite. It follows from the inflation-restriction exact sequence that $\ker(h)=\H^1(\Gamma,\H^0(\bQ_\cyc,E[p^\infty]\otimes\chi))=\H^0(\bQ_\cyc,E^{K_\chi}[p^\infty])_\Gamma$, where the identification comes from~\cite[Lemma B.2.8]{RubinEulerSystems}, but the numerator is already finite by Ribet's theorem. Now that $\H^0(\bQ_\cyc,E^{K_\chi}[p^\infty])$ is finite, $\vert\H^0(\bQ_\cyc,E^{K_\chi}[p^\infty])_\Gamma\vert=\vert\H^0(\bQ_\cyc,E^{K_\chi}[p^\infty])^\Gamma\vert=\vert E^{K_\chi}(\bQ)[p^\infty]\vert$. Combining these, we see that\[
\vert S_{E\otimes\chi}^\Gamma\vert=\vert \Sel_{\Gr_0}(E\otimes\chi/\bQ)\vert \vert\ker(g)\vert/\vert E^{K_\chi}(\bQ)[p^\infty]\vert.
\]

It remains to compute $\ker(g)$. We will first turn to study $\ker(r)=\prod_{v\in\Sigma}\ker(r_v)$ where where  {\small\begin{align*}r_p:\frac{\H^1(\bQ_p,E[p^\infty]\otimes\chi)}{\im(\H^1(\bQ_p,\Fil^+(E[p^\infty])\otimes\chi)\to\H^1(\bQ_p,E[p^\infty]\otimes\chi))_\div}\to\\\frac{\H^1(\bQ_{\cyc,p},E[p^\infty]\otimes\chi)}{\ker(\H^1(\bQ_{\cyc,p},\Fil^+(E[p^\infty])\otimes\chi)\to\H^1(\bQ_{\cyc,p},E[p^\infty]\otimes\chi))},\end{align*}}
and for $v\in\Sigma-\{p\}$, \[r_v:\H^1(\bQ_v,E[p^\infty]\otimes\chi)\to\H^1(\bQ_{\cyc,v},E[p^\infty]\otimes\chi).\]
At a place $v\ne p$, $\ker(r_v)$ is identified with $\ker(\H^1(\bQ_p,E^{K_\chi})\to\H^1(\bQ_{\cyc,v},E^{K_\chi}[p^\infty]))$. From~\cite[Lemma 3.3]{Greenberg1999} and the discussion after, $\ker(r_v)=0$ if $E^{K_\chi}$ has good reduction at $v$, and $\vert\ker(r_v)\vert=\ord_p(c_{v,E^{K_\chi}})$ where $E^{K_\chi}$ has bad reduction. Here $c_{v,E^{K_\chi}}$ is the Tamagawa number factor for $E^{K_\chi}$ at $v$.

At $p$, we compute $\ker(r_p)$ by going to $K_\chi$, noting that $\ker(r_{p,E/K_{\chi}})$ decomposes into $\ker(r_{p,E})\oplus \ker(r_{p,E\otimes\chi})$, so $\vert \ker(r_{p,E\otimes\chi})\vert=\vert\ker(r_{p,E/K_\chi})\vert/\vert \ker(r_{p,E})\vert$. Now we consider the cases where $E$ has good ordinary reduction, non-split multiplicative reduction and split-multiplicative reduction separately. It suffices to study~\cite[Proposition 2.5 and Lemma 3.4]{Greenberg1999} in these cases.
\subsection*{Computing $\ker(r_p)$}
$ $\newline
The reference in this subsection are all to~\cite{Greenberg1999}.

\begin{itemize}
\item When $E$ has good ordinary reduction:
\end{itemize}
In this case, Lemma 3.4 implies that $\vert\ker(r_{p,E/K_\chi})\vert=\vert\tilde{E}(\F)_p\vert^2=\vert\ker(r_{p,E})\vert$, as $(K_\chi)_p$ and $\bQ_p$ share the same residue field $\F$. Thus $\vert\ker(r_p)\vert=\vert\ker(r_{p,E\otimes\chi})\vert=1$.

\begin{itemize}
	\item When $E$ has non-split multiplicative reduction:
\end{itemize}
In the multiplicative reduction case, the remark after Proposition 2.5 shows that $\im(\kappa_{\bQ_p})=\im(\eta_{\bQ_p})$ and $\im(\kappa_{(K_\chi)_p})=\im(\eta_{(K_\chi)_p})$, so the error terms in Proposition 2.5 become trivial. The computation in Lemma 3.4 still works since $\Fil^+(E[p^\infty])\ne\bQ_p/\bZ_p(1)$, so $\ker(r_{p,E})=\ker(b_{p,E})\iso\ker(d_{p,E})$ having size $\vert\tilde{E}(\F)[p^\infty]\vert$. But $(K_\chi)_p$ and $\bQ_p$ have the same residual field $\F$, so $\ker(r_{p,E/K_\chi})$ has size $\vert\tilde{E}(\F)[p^\infty]\vert$ as well. Therefore, $\vert\ker(r_p)\vert=\vert\ker(r_{p,E\otimes\chi})\vert=1$.
\begin{itemize}
	\item When $E$ has split multiplicative reduction:
\end{itemize}
As before, the error terms in Proposition 2.5 is trivial. It suffices to compare $\ker(b_{p,E})$ with $\ker(b_{p,E/K_\chi})$ in Lemma 3.4. Now the first row of sequence (5) is no longer surjective because $\H^2(F_p,\Fil^+(E[p^\infty]))=\H^2(F_p,\bQ_p/\bZ_p(1))$ is dual to $\H^0(F_p,\bZ_p))=\bZ_p$ for $F=\bQ$ or $K_\chi$, so $\H^2(F_p,\Fil^+(E[p^\infty]))=\bQ_p/\bZ_p$. However, the next term $\H^2(F_p,E[p^\infty])$ remains $0$ since it's dual to $\H^0(F_p,T_pE)=0$ as $T_pE$ is a non-split extension of $\bZ_p$ by $\bZ_p(1)$. Therefore we have a commutative diagram\begin{center}{\small\begin{equation*}\begin{tikzcd}	
				0\ar[r]&\ker(b_F)\ar[r]\ar[d]&\ker(d_F)\ar[d]&&			\\
				0\ar[r]&\im(\pi_p) \ar[r]\ar[d,"b_F"]& \H^1(F_p,\Fil^-(E[p^\infty]))\ar[d,"d_F"]\ar[r] &\H^2(F_p,\Fil^+(E[p^\infty]))\ar[d,"e_F"]\ar[r]&0\\
				0\ar[r]&\im(\pi_\eta)\ar[r]&\H^1(F_{\cyc,p},\Fil^-(E[p^\infty]))\ar[r]&\H^2(F_{\cyc,p},\Fil^+(E[p^\infty]))\ar[r]&0,\\	 	\end{tikzcd}
	\end{equation*}}
\end{center}
where $\ker(d_F)\iso \H^1(F_{\cyc,p}/F_p,\H^0(F_{\cyc,p},\bQ_p/\bZ_p)))\iso\bQ_p/\bZ_p$ from the inflation-restriction sequence since $\Fil^-(E[p^\infty])=\bQ_p/\bZ_p$, for both $F=\bQ$, $K_\chi$. As $\ker(d_{E/K_\chi})\iso\ker(d_E)\oplus\ker(d_{E\otimes\chi})$, it follows that $\ker(d_{E\otimes\chi})$, and hence $\ker(b_{p,E\otimes\chi})$ is trivial. Then $\vert\ker(r_p)\vert=\vert\ker(r_{p,E\otimes\chi})\vert=1$.

Thus in any case, we have thus proved \[\vert\ker(r)\vert=\prod_{p\ne v\mid N(E^{K_\chi})}\ord_p(c_{v,E^{K_\chi}}),\]where $N(E^{K_\chi})$ is the product of bad primes for $E^{K_\chi}$. It remains to relate $\ker(r)$ to $\ker(g)$, following~\cite[section 4]{Greenberg1999}.

\subsection*{Computing $\ker(g)$}
$ $\newline
We continue to make frequent use of results in~\cite[section 4]{Greenberg1999} for both $\bQ$ and $K_\chi$, and keep references from \emph{loc. cit.}. First, Lemma 4.6 implies that the global-to-local map\[
\H^1(K_\chi^\Sigma/(K_\chi)_\cyc,E[p^\infty])\to P^\Sigma_E((K_\chi)_\cyc)
\]
is surjective, hence the surjectivity of \[
\H^1(\bQ^\Sigma/\bQ_\cyc,E[p^\infty]\otimes\chi)\to P^\Sigma_{E\otimes\chi}(\bQ_\cyc)
\]
follows from taking the $-$-isotypic component.

With this surjectivity, the pure cohomological argument in the rest of section $4$ (relying only on the divisibility of $E[p^\infty]$, which still holds for $E[p^\infty]\otimes\chi$, and the self-duality of Kummer image, which still holds for the potentially semistable curve $E^{K_\chi}$ by~\cite[Proposition I.6.7, Remark I.7.1]{RubinEulerSystems}), yields the analog of Lemma 4.7:
\begin{lemma}
	Under the assumption of~\cref{fE(0)}, we have\[
	\vert\ker(g)\vert=\vert\ker(r)\vert\vert(S_{E\otimes\chi})_\Gamma\vert/\vert E^{K_\chi}(\bQ)[p^\infty]\vert.
	\]
\end{lemma}
Note that even though $\Fil^+(E[p^\infty])$ could possibly be $\bQ_p/\bZ_p(1)$, the error by an extra $\delta=1$ in~\cite[Lemma 2.3]{Greenberg1999} will disappear after passing to $\bQ_\cyc$. The proof of~\cref{fE(0)} now follows from the combination of the above lemmas, equations and the identification $\Sel_{\Gr_0}(E\otimes\chi/\bQ)=\Sel(E^{K_\chi}/\bQ)$.

\section{The Main Consequences}
\subsection{Greenberg's $\mu=0$ conjecture for certain potentially ordinary primes}
In~\cref{4}, we showed that Greenberg's conjecture is true for any ramified quadratic twists of an ordinary elliptic curve at an Eisenstein prime $p\geq 3$. Thus we have proved the following.
\begin{theorem}\label{Grpotord}
	Let $E$ be an elliptic curve with potentially ordinary reduction at a prime $p\geq 3$ with $E[p]$ reducible, such that it is a quadratic twist of an elliptic curve ordinary at $p$. Then Greenberg's $\mu=0$ conjecture is true, i.e., there exists an elliptic curve $E'$ in its $\bQ$-isogeny class with $\mu(\Sel(E'/\bQ_\cyc)^\vee)=0$. In particular, Greenberg's $\mu=0$ conjecture is true for all elliptic curves of potentially multiplicative reduction at an Eisenstein prime $p\geq 3$.
\end{theorem}
\begin{proof}
	This is a combination of the main results in~\cref{4} with~\cref{preliminary}.
\end{proof}

The proof of Greenberg's $\mu=0$ conjecture at Eisenstein primes has an interesting application. Assuming $E[p]$ is reducible over $G_\bQ$, then the proof of~\cref{Grpotord} in~\cref{4} in fact applies to Selmer groups over an abelian field $F$ if one replaces $L=\bQ(E[p])K_\chi$ there also with $F(E[p])K_\chi$, provided that $E[p]$ remains a non-split extension of two characters over $F$. Indeed, by the same logic, the non-splitness of the extension with the vanishing of $\mu$ for the field $F(E[p])K_\chi$ which is again either a trivial or a $p$-extension of an abelian number field $F(\phi,\psi)K_\chi=F(\phi,\psi,\chi)$ yields the desired result.

In fact, $E[p]$ always ramains non-split over $F$ if it is non-split over $\bQ$ since the latter is equivalent to $\bQ(E[p])$ being a strict $p$-extension of $\bQ(\phi,\psi)$, and the same would hold with $\bQ$ replaced by an abelian number field $F$. So this is always possible for at least one $E$ in its isogeny class by Ribet's lemma. Note that the proof relies solely on analyzing the mod-$p$ Selmer groups, and we can in turn gather information on the usual $p^\infty$-Selmer groups with $E[p^\infty]$ coefficients.

\begin{corollary}\label[corollary]{Mazur}
	Assume $E/\bQ$ has ordinary reduction at an Eisenstein prime $p$ and $\chi$ is a quadratic ramified character. Then the Selmer group for the quadratic twist $E^\chi$ is $\Lambda$-cotorsion over any abelian number field $F$, i.e., Mazur's cotorsion conjecture holds in these cases.
\end{corollary} 
\begin{proof}
	Suppose not, then $\Sel(E/F_\cyc)[p]$ would have infinite $\F_p$-dimension. But $\Sel(E/F_\cyc)[p]$ is pseudo-isomorphic to $\Sel(E[p]/F_\cyc)$. This contradicts our $\mu=0$ result for an isogenous curve, since $\Lambda$-cotorsionness is invariant under isogeny.
\end{proof}
For ordinary curves, more general results have been obtained by results of Kato and Rohrlich, using Euler system. For potentially ordinary curves, when $F=\bQ$, similar results are shown in~\cite{Del98}.

\subsection{$p$-part BSD formulae at certain potentially ordinary primes}
We prove the $p$-part of BSD formulae for some elliptic curves with potentially ordinary reduction that are quadratic twists of ordinary curves.

\subsubsection{$p$-part BSD in rank $0$}
We now compare Kato's interpolation formulae of different branches of the $p$-adic $L$-function to the formula in~\cref{fE(0)}, given the Iwasawa Main Conjecture~\cref{IMC}. Recall that from~\cref{3}, it suffices to evaluate Kato's $p$-adic $L$-function (\cite[Theorem 16.2(i)]{Kato}) at $\omega^{\frac{p-1}{2}}$.

\begin{theorem}
	\[\cL(E,\omega^{\frac{p-1}{2}})(0)\sim\frac{p}{\alpha G(\omega^{\frac{p-1}{2}},\zeta_p)}\cdot\frac{L(f,\omega^{\frac{p-1}{2}},1)}{\Omega_{\pm}}\]
	up to a $p$-adic unit. Here $\alpha$ is the unit root of the Hecke polynomial of $E$ if $E$ has good ordinary reduction, and $\alpha=a_p$ if $E$ has multiplicative reduction, and $\pm=\omega^{\frac{p-1}{2}}(-1)=(-1)^{\frac{p-1}{2}}$.
\end{theorem}
\begin{proof}
	It is just~\cite[Theorem 16.2(i)]{Kato} if we take $\chi=\omega^{\frac{p-1}{2}}$ (so $n=1$), $k=2$, and $r=1$.
\end{proof}

Now under the assumption that $E$ has ordinary reduction, the Manin constant is a $p$-adic unit, so $\Omega_{\pm}\sim\Omega^\pm_{E}$ up to a $p$-adic unit. $\alpha$ is by definition a $p$-adic unit as well. The Gauss sum is evaluated to be\[
	G(\omega^{\frac{p-1}{2}},\zeta_p)=
\begin{cases}
\sqrt{p} &\text{if }p\equiv 1\pmod{4}\\
i\cdot\sqrt{p} &\text{if }p\equiv 3\pmod{4}.
\end{cases}
\]
Thus the formula becomes\[
\cL(E,\omega^{\frac{p-1}{2}})(0)\sim \begin{cases}
	\sqrt{p}\cdot \frac{L(E\otimes\omega^{\frac{p-1}{2}},1)}{\Omega_E^+} &\text{if }p\equiv 1\pmod{4}\\
	(-i)\cdot\sqrt{p}\cdot \frac{L(E\otimes\omega^{\frac{p-1}{2}},1)}{\Omega_E^-} &\text{if }p\equiv 3\pmod{4}.
\end{cases}
\]
We now compare the period of $E$ to the period of $E\otimes\omega^{\frac{p-1}{2}}$ using~\cite[Propoition 2.5]{Pal11}. Since twisting by $\omega^{\frac{p-1}{2}}$ corresponds to twisting by \[
\begin{cases}
	\bQ(\sqrt{p})\text{ and }d=p &\text{if }p\equiv 1\pmod{4}\\
	\bQ(\sqrt{-p})\text{ and }d=-p &\text{if }p\equiv 3\pmod{4},
\end{cases}\]
and if $E$ has good ordinary reduction, $p\nmid \Delta$, while if $E$ has multiplicative reduction, $p\nmid c_4$, the $\tilde{u}$ from \emph{loc.cit.} is a $p$-adic unit and we have\[
\Omega_{E^\otimes\omega^{\frac{p-1}{2}}}\sim\begin{cases}
	\frac{\Omega_E}{\sqrt{p}} &\text{if }p\equiv 1\pmod{4}\\
	\frac{\Omega^-_E}{i\cdot\sqrt{p}} &\text{if }p\equiv 3\pmod{4},
\end{cases}
\]
we thus arrive at the final formula\begin{equation}\label{Lptwist}
\cL(E,\omega^{\frac{p-1}{2}})(0)\sim \frac{L(E\otimes\omega^{\frac{p-1}{2}},1)}{\Omega_{E\otimes\omega^{\frac{p-1}{2}}}}.
\end{equation}

Combining the above formula with~\cref{IMC} and~\cref{fE(0)}, we have the following theorem.

\begin{theorem}\label{rank0BSD}
	Let $E$ be an elliptic curve of ordinary reduction at $p\geq 3$ such that $E[p]$ is reducible, and let $\chi$ be an admissible quadratic ramified character corresponding to a quadratic field $K_\chi$. Assume that $L(E^{K_\chi},1)\ne 0$ for the quadratic twist $E^{K_\chi}$, and that $\#E'(\bQ)_\tors=1$ for all curve $E'$ isogenous to $E^{K_\chi}$. Then\[
	\ord_p\bigl(\frac{L(E^{K_\chi},1)}{\Omega_{E^{K_\chi}}}\bigr)=\ord_p\bigl(\#\Sha(E^{K_\chi}/\bQ)\cdot\Tam(E^{K_\chi}/\bQ)\bigr),
		\]
	i.e., the rank $0$ $p$-part BSD formula for $E^{K_\chi}$ holds. In particular, rank $0$ $p$-part BSD formula holds for all elliptic curves of potentially ordinary reduction at $p$ coming form such twists, including all potentially multiplicative elliptic curve at $p=3$.
\end{theorem}
\begin{proof}
	As before, it suffices to consider the case where $\chi=\omega^{\frac{p-1}{2}}$. If $L(E^{K_\chi},1)\ne 0$, Kolyvagin's theorem implies that $\rk_\bZ(E^{K_\chi}(\bQ))=0$ and $\#\Sha(E^{K_\chi}/\bQ)<\infty$. In particular, $\#\Sel(E^{K_\chi}/\bQ)=\#\Sha(E^{K_\chi}/\bQ)[p^\infty]<\infty$. One then deduces the formula from a combination of~\eqref{Lptwist},~\cref{IMC} and~\cref{fE(0)}, by noting that $p$ cannot divide the Tamagawa number at $p$ in the case of potentially ordinary reduction (for $p=3$, such a divisibility can happen for Kodaira type IV and IV*, both of which are of potentially supersingular reduction). If an elliptic curve has potentially multiplicative reduction, it always comes from a quadratic ramified twist (but not necessarily admissible) of an ordinary elliptic curve. In particular, for $p=3$, all potentially multiplicative reduction such that $\F_3(\mathbf{1})$ is not in $E[3]^{ss}$ arise in this way, since all such twists are admissible, as no character can be tamely ramified and even.
\end{proof}

\printbibliography

@article {CGS,
	KEY = {CGS25},
	AUTHOR = {Castella, Francesc and Grossi, Giada and Skinner, Christopher},
	TITLE = {{Mazur's main conjecture at Eisenstein primes}},
	JOURNAL = {Mathematische Annalen},
	VOLUME = {393},
	YEAR = {2025},
	PAGES = {2451--2506},
	DOI = {10.1007/s00208-025-03239-x},
	URL = {https://doi.org/10.1007/s00208-025-03239-x},
}

@article{Del98, title={{Iwasawa Theory for Elliptic Curves at Unstable Primes}}, volume={113}, DOI={10.1023/A:1000408925932}, number={2}, journal={Compositio Mathematica}, author={DELBOURGO, DANIEL}, year={1998}, pages={123–154}}

@article {EPW06,
	AUTHOR = {Emerton, Matthew and Pollack, Robert and Weston, Tom},
	TITLE = {{Variation of Iwasawa invariants in Hida families}},
	JOURNAL = {Inventiones Mathematicae},
	VOLUME = {163},
	YEAR = {2006},
	PAGES = {523--580},
	DOI = {10.1007/s00222-005-0467-7},
	URL = {https://doi.org/10.1007/s00222-005-0467-7},
}

@article {Greenberg1989,
	AUTHOR = {Greenberg, Ralph},
	TITLE = {Iwasawa Theory for $p$-adic Representations},
	JOURNAL = {Advanced Studies in Pure Mathematics},
	YEAR = {1989},
	VOLUMN = {17},
	PAGES = {97--137},
	MRNUMBER = {1097613},
    DOI = {10.2969/aspm/01710097},
    URL = {https://doi.org/10.2969/aspm/01710097},
}

@Inbook{Greenberg1999,
author={Greenberg, Ralph},
editor={Viola, Carlo},
title={Iwasawa theory for elliptic curves},
bookTitle={Arithmetic Theory of Elliptic Curves: Lectures given at the 3rd Session of the Centro Internazionale Matematico Estivo (C.I.M.E.) held in Cetraro, Italy, July 12--19, 1997},
year={1999},
publisher={Springer Berlin Heidelberg},
address={Berlin, Heidelberg},
pages={51--144},
isbn={978-3-540-48160-7},
doi={10.1007/BFb0093453},
url={https://doi.org/10.1007/BFb0093453"},
}

@article {GV00,
	AUTHOR = {Greenberg, Ralph and Vatsal, Vinayak},
	TITLE = {{On the Iwasawa invariants of elliptic curves}},
	JOURNAL = {Inventiones mathematicae},
	VOLUME = {142},
	YEAR = {2000},
	PAGES = {17--63},
	DOI = {10.1007/s002220000080},
	URL = {https://doi.org/ 10.1007/s002220000080},
}

@article {Kato,
	AUTHOR = {Kato, Kazuya},
	TITLE = {$p$-adic Hodge theory and values of zeta functions of modular forms},
    JOURNAL = {Ast\'{e}risque},
    FJOURNAL = {Ast\'{e}risque},
	VOLUME = {295},
	YEAR = {2004},
	PAGES = {117--290},
}

@misc{KY24,
      title={On the anticyclotomic Iwasawa theory of newforms at Eisenstein primes of semistable reduction}, 
      author={Timo Keller and Mulun Yin},
      year={2024},
      eprint={2402.12781},
      archivePrefix={arXiv},
      primaryClass={math.NT},
      url={https://arxiv.org/abs/2402.12781}, 
}

@misc{KY24b,
      title={$p$-converse theorems for elliptic curves of potentially good ordinary reduction at Eisenstein primes}, 
      author={Timo Keller and Mulun Yin},
      year={2024},
      eprint={2410.23241},
      archivePrefix={arXiv},
      primaryClass={math.NT},
      url={https://arxiv.org/abs/2410.23241}, 
}

@article{LY26,
      title={On Greenberg's conjecture for rational elliptic curves at Eisenstein primes}, 
      author={Zichao Lin and Mulun Yin},
      year={2026},
      eprint={2609.40282},
      archivePrefix={arXiv},
      primaryClass={math.NT}
}

@article{MazurWiles,
	title={{Class fields of abelian extensions of Q}},
	author={Mazur, Barry and Wiles, Andrew},
	journal={Inventiones Mathematicae},
	volume={76},
	issue={2},
	pages={179-330},
	year={1984},
}

@article{Mazur,
	title={{Rational isogenies of prime degree (with an appendix by D. Goldfeld)}},
	author={Mazur, Barry},
	journal={Inventiones Mathematicae},
	volume={44},
	pages={129-162},
	year={1978},
}

@article{MSD74,
    author = {Mazur, Barry and Swinnerton-Dyer, Peter},
    title = {{Arithmetic of Weil curves}},
    journal = {Invent. Math.},
    year = {1974},
    volume = {25},
    pages = {1-61},
}

@article{MTT,
    author = {Mazur, Barry and Tate, John and Teitelbaum, Jeremy},
    title = {{On p-adic analogues of the conjectures of Birch and Swinnerton-Dyer}},
    journal = {Invent. Math.},
    year = {1986},
    issue = {84},
    pages = {1-48},
}

@article {NekovarSelmerComplexes,
    AUTHOR = {Nekov\'{a}\v{r}, Jan},
     TITLE = {Selmer complexes},
   JOURNAL = {Ast\'{e}risque},
  FJOURNAL = {Ast\'{e}risque},
    NUMBER = {310},
      YEAR = {2006},
     PAGES = {viii+559},
      ISSN = {0303-1179},
      ISBN = {978-2-85629-226-6},
   MRCLASS = {11R23 (11F41 11G40 11R34 22E41)},
  MRNUMBER = {2333680},
MRREVIEWER = {Laurent N. Berger},
}

@article{Pal11,
author = {Pal, Vivek},
year = {2011},
month = {09},
pages = {1513-1525},
title = {Periods of quadratic twists of elliptic curves},
volume = {140},
journal = {Proceedings of the American Mathematical Society},
doi = {10.1090/S0002-9939-2011-11014-1}
}

@article {PollackWeston2011,
    AUTHOR = {Pollack, Robert and Weston, Tom},
     TITLE = {On anticyclotomic {$\mu$}-invariants of modular forms},
   JOURNAL = {Compos. Math.},
  FJOURNAL = {Compositio Mathematica},
    VOLUME = {147},
      YEAR = {2011},
    NUMBER = {5},
     PAGES = {1353--1381},
      ISSN = {0010-437X},
   MRCLASS = {11R23 (11F33 11F80)},
  MRNUMBER = {2834724},
MRREVIEWER = {Andrea Mori},
       DOI = {10.1112/S0010437X11005318},
       URL = {https://doi.org/10.1112/S0010437X11005318},
}

@article {Ribet,
	AUTHOR = {Ribet, Kenneth A.},
	TITLE = {Torsion points on abelian varieties in cyclotomic extensions (appendix to an article by N. Katz and S. Lang)},
	JOURNAL = {Enseign. Math.},
	VOLUME = {27},
	ISSUE = {2},
	YEAR = {1982},
	NUMBER = {3-4},
	PAGES = {285--319},
}

@book {RubinEulerSystems,
    AUTHOR = {Rubin, Karl},
     TITLE = {Euler systems},
    SERIES = {Annals of Mathematics Studies},
    VOLUME = {147},
      NOTE = {Hermann Weyl Lectures. The Institute for Advanced Study},
 PUBLISHER = {Princeton University Press, Princeton, NJ},
      YEAR = {2000},
     PAGES = {xii+227},
      ISBN = {0-691-05075-9; 0-691-05076-7},
   MRCLASS = {11R23 (11G40 11R34 11R42)},
  MRNUMBER = {1749177},
MRREVIEWER = {Jan Nekov\'{a}\v{r}},
       DOI = {10.1515/9781400865208},
       URL = {https://doi.org/10.1515/9781400865208},
}

@article{Vat99,
author = {V. Vatsal},
title = {{Canonical periods and congruence formulae}},
volume = {98},
journal = {Duke Mathematical Journal},
number = {2},
publisher = {Duke University Press},
pages = {397 -- 419},
year = {1999},
doi = {10.1215/S0012-7094-99-09811-3},
URL = {https://doi.org/10.1215/S0012-7094-99-09811-3}
}

@article{Wut17,
 TITLE = {{On the integrality of modular symbols and Kato's Euler system for elliptic curves}},
 AUTHOR = {Wuthrich, Christian},
YEAR = {2014},
MONTH = {01},
 JOURNAL = {Documenta Mathematica},
VOLUME = {19},
PAGES = {381 - 402},
DOI = {10.4171/dm/450},
}

@misc{XY26,
      title={{Kato's Iwasawa main conjecture for modular forms at Eisenstein primes}}, 
      author={Xu, Ruichen and Yin, Mulun},
      year={2026},
      note={in preparation},
}

\end{document}